\documentclass{amsart}
\usepackage{amssymb, latexsym}
\usepackage{enumerate}
\usepackage{footnote}
\usepackage{mathtools,color}

\theoremstyle{plain}
\newtheorem{thm}{Theorem}

\newtheorem{rem}[thm]{Remark}

\newtheorem{prop}[thm]{Proposition}

\numberwithin{equation}{section} \numberwithin{thm}{section}

\begin{document}

\title[Cubic wave equation on $\mathbb{T}^{2}$]
{On the interpolation with the potential bound for global solutions of the defocusing cubic wave equation on
$\mathbb{T}^{2}$}

\author{Tristan Roy}

\address{Universit\'e de Cergy-Pontoise and Nagoya University }
\email{tristanroy@math.nagoya-u.ac.jp}


\vspace{-0.3in}

\begin{abstract}

Consider the solutions of the defocusing cubic wave equation
\begin{align*}
\partial_{tt} u  - \Delta u  & = -u^{3},
\end{align*}
with real data in $ H^{s} (\mathbb{T}^{2}) \times H^{s-1} (\mathbb{T}^{2})$, $s > \frac{2}{5}$. We prove that the solutions exist globally in time
by contradiction. Assuming that one of the maximal times of existence is finite, we prove that the Sobolev norm of each of these solutions is bounded
in an open neighborhood of it by using the I-method (see e.g \cite{almckstt}). It consists of introducing a mollified energy and estimating its growth
in this neighborhood. The growth is first estimated on small subintervals by means of dispersive bounds; then it is estimated on the whole interval by
an iteration process. In order to enlarge the size of subintervals on which the dispersive bounds hold, we interpolate with the potential bound for
the low frequency part. Our results complement those obtained in \cite{bourgbook}.
\end{abstract}

\maketitle
\tableofcontents

\section{Introduction}

In this paper we study the defocusing cubic wave equation on the torus $\mathbb{T}^{2} := \mathbb{R}^{2} / \mathbb{Z}^{2}$, i.e

\begin{equation}
\partial_{tt} u - \triangle u = -u^{3},
\label{Eqn:DefocCubWave}
\end{equation}
with data $ (u(0),\partial_{t} u(0)):= (u_{0},u_{1}) \in H^{s}(\mathbb{T}^{2}) \times H^{s-1} (\mathbb{T}^{2})$, $s \in \mathbb{R}$. Here $H^{s}(\mathbb{T}^{2})$ denotes
the standard inhomogeneous Sobolev space, i.e the closure of smooth functions with respect to the
following norm

\begin{equation}
\| f \|_{H^{s}(\mathbb{T}^{2})} := \left\| \langle n \rangle^{s} \hat{f}(n) \right\|_{l^{2}(\mathbb{Z}^{2})} \cdot
\nonumber
\end{equation}
Here $\hat{f}$ denotes the Fourier transform of $f$, i.e

\begin{equation}
\hat{f}(\xi) := \int_{\mathbb{T}^{2}} f(x) e^{-i 2 \pi \xi \cdot x} \, dx.
\nonumber
\end{equation}
It is well-known (see e.g \cite{linsog}) that given data $(u_{0},u_{1}) \in H^{s}(\mathbb{T}^{2}) \times H^{s-1}(\mathbb{T}^{2})$, $s > \frac{1}{4}$, there exist a time of local existence $T_{l} > 0$ and a unique solution $ (u,\partial_{t} u) $ that lies in a subset of $ \mathcal{C} \left( [0,T_{l}], H^{s}(\mathbb{T}^{2}) \right) \times \mathcal{C} \left( [0,T_{l}],  H^{s-1}(\mathbb{T}^{2})  \right) $ and that satisfies in the distributional sense

\begin{equation}
\begin{array}{l}
u(t) = \cos{(tD)} u_{0}  + \frac{\sin{(tD)}}{D} u_{1}
 - \int_{0}^{t} \frac{\sin((t-t')D)}{D} u^{3}(t') \, dt'.
\end{array}
\nonumber
\end{equation}
Here $D$ denotes the multiplier defined in the Fourier domain by $\widehat{Df}(\xi):= |\xi| \hat{f}(\xi)$. By iterating the process, one can extend the
domain of the solution to the maximal interval of existence $I_{max}:= (T_{-},T_{+})$. Next we would like to know whether we can
choose the maximal times of existence $T_{+}$ and $-T_{-}$ to be infinite. In other words, does the solution exist globally in time? Since the time of local existence depends on the norm of the initial data, it is sufficient to find for all time $T > 0$ a finite \textit{a priori} bound of the $H^{s} \times H^{s-1}$ norm of the solution. The global existence for real data $ (u_{0},u_{1}) \in H^{s} (\mathbb{T}^{2}) \times H^{s-1} (\mathbb{T}^{2})$, $s \geq 1$, follows immediately from the conservation of the energy

\begin{equation}
E(u) := \frac{1}{2} \int_{\mathbb{T}^{2}} |\partial_{t} u(t,x)|^{2} \, dx + \frac{1}{2} \int_{\mathbb{T}^{2}} |\nabla u(t,x)|^{2} \, dx +
\frac{1}{4} \int_{\mathbb{T}^{2}} |u(t,x)|^{4} \, dx.
\nonumber
\end{equation}
It remains to understand the global behavior of the solution with real data $ (u_{0},u_{1}) \in H^{s}(\mathbb{T}^{2}) \times H^{s-1}(\mathbb{T}^{2})$, $s < 1$. Bourgain proved in \cite{bourgbook} that the solution with real data $ (u_{0},u_{1}) \in H^{s}(\mathbb{T}^{2}) \times H^{s-1}(\mathbb{T}^{2})$, $s \geq \frac{1}{2}$, exists globally in time. Our main theorem is the following:
\vspace{5mm}
\begin{thm}
The solution of the defocusing cubic wave equation with real data  $(u_{0},u_{1}) \in H^{s}(\mathbb{T}^{2}) \times H^{s-1}(\mathbb{T}^{2})$, $s > \frac{2}{5}$, exists globally in time.
\end{thm}

\vspace{5mm}
\begin{rem}
A straightforward modification of the arguments in this paper shows
that the theorem also holds for solutions of the equation $\partial_{tt} u -\triangle u = - |u|^{2}u $ with complex-valued data
$(u_{0},u_{1}) \in H^{s}(\mathbb{T}^{2}) \times H^{s-1}(\mathbb{T}^{2})$, $s > \frac{2}{5}$.
\end{rem}
\begin{rem}
In the sequel of this paper we assume implicitly that the data is real.
\end{rem}

Now we explain the main ideas and how this paper is organized. Since the energy may be infinite for data $(u_{0},u_{1}) \in
H^{s}(\mathbb{T}^{2}) \times H^{s-1}(\mathbb{T}^{2})$, one cannot use it to control the $H^{s} \times H^{s-1}$ norm of the solution at a given time. Instead we introduce the mollified energy

\begin{equation}
E(Iu(t)) := \frac{1}{2} \int_{\mathbb{T}^{2}} \left( \partial_{t} I u(t,x) \right)^{2} \, dx +  \frac{1}{2} \int_{\mathbb{T}^{2}} |\nabla I u(t,x)|^{2} \, dx + \frac{1}{4} \int_{\mathbb{T}^{2}} \left( I u(t,x) \right)^{4} \, dx .
\nonumber
\end{equation}
Here the multiplier $I$ is defined in the Fourier domain by

\begin{equation}
\widehat{If}(\xi) = m(\xi) \hat{f}(\xi),
\end{equation}
with $m(\xi):= \eta \left( \frac{\xi}{N} \right)$, $N >> 1$ a parameter to be chosen, and $\eta$ such that

\begin{equation}
\left\{
\begin{array}{l}
\eta (\xi) \; radial \; smooth, \; decreasing, \; 0 \leq  \eta \leq 1 \\ \\
\eta(\xi) := \left\{
\begin{array}{l}
1, \, |\xi| \leq 1 \\
\frac{1}{|\xi|^{1-s}}, \, |\xi| \geq  2 :
\end{array}
\right.
\end{array}
\right.
\nonumber
\end{equation}
this is the $I$-method (see e.g \cite{almckstt}), inspired from the \textit{Fourier truncation method} (see \cite{bourgbook}).
This mollified energy is finite if the data is in $H^{s}(\mathbb{T}^{2}) \times H^{s-1} (\mathbb{T}^{2})$. Furthermore it has a slow variation.
So we can control the $H^{s}(\mathbb{T}^{2}) \times H^{s-1}(\mathbb{T}^{2})$ norm of the solution on a time interval by estimating the size of the mollified
energy on the same interval. Assuming towards a contradiction that one of the maximal times of existence is finite, we control the variation
of the mollified energy in an open neighborhood of it, which yields a contradiction. It is well-known that the better we estimate this variation,
the rougher the data is for which the solution exists globally in time. The variation is first estimated on small subintervals on which some dispersive bounds
hold. Then it is estimated on the whole interval by an iteration process. Therefore the larger the subintervals are on which the dispersive bounds hold, the better the estimate of the variation is on the whole interval. In Section \ref{Section:Gwp12}, we prove that the solution exists globally for
$s > \frac{4}{9}$. In order to maximize the size of the subintervals for which the dispersive bounds hold, we adopt the following strategy:

\begin{itemize}

\item we first decompose the solution into its low frequency part and its high frequency part

\item then we plug this decomposition into the nonlinearity;

\item then we use three types of bounds (the potential bound, the kinetic bounds, and the dispersive bounds) and we estimate the nonlinearity
by using the following scheme: for the low frequency part, we mostly interpolate with the potential bound, stronger than the kinetic bounds; for the high frequency part, we mostly use the dispersive bounds, using to our advantage the regularity the solution:

\end{itemize}
see Subsection \ref{Subsec:LocalBd} for more details. In Section \ref{Section:Gwpbelow12}, we prove that the solution exists globally for $s > \frac{2}{5}$. The proof uses an adapted-linear decomposition \cite{troy}, based upon the fact that the nonlinear part of the solution is smoother than $H^{s}$
(a property observed in \cite{bourgbook}). A key element of the proof is to maximize the size of the subintervals for which this gain of regularity
holds. Again we use the same strategy to achieve this goal. The use of the potential bound in Section \ref{Section:Gwp12} and in Section \ref{Section:Gwpbelow12} is crucial: we shall explain it in Section \ref{Sec:Bounds} and in the references given in this section.

\begin{rem}
In \cite{bourgbook}, Bourgain proved the global existence of the solutions of (\ref{Eqn:DefocCubWave}) with data in $H^{s} (\mathbb{T}^{2})
\times H^{s-1}(\mathbb{T}^{2})$, $s \geq \frac{1}{2}$, by controlling the variation of

\begin{equation}
\tilde{E}(\tilde{v}) := \frac{1}{2} \int_{\mathbb{T}^{2}} |D \tilde{v}(t,x)|^{2} \, dx
+ \frac{1}{4} \int_{\mathbb{T}^{2}} | \Re(\tilde{v} (t,x))|^{4} \, dx,
\end{equation}
with $\tilde{v}(t) :=  v(t) - e^{-i t D} P_{> N} v(0)$, $v(t) := u(t) + i \frac{\partial_{t} u(t)}{D}$, and $N \gg 1$ parameter to be chosen. This fact, combined with a nonsqueezing property of balls in cylinders of smaller width, allowed him to prove a weak turbulence property
for solutions with data in the symplectic space $H^{\frac{1}{2}} (\mathbb{T}^{2}) \times H^{-\frac{1}{2}}(\mathbb{T}^{2})$: see p 96. The variation is estimated by means of an $L_{t}^{4} L_{x}^{4} ([0,1])$ bound of the solutions of the linear wave equation with data of which the Fourier modes are localized
\footnote{Hence the Fourier modes of the solutions of the linear wave equation are localized for all time}. We do not expect this bound to hold for the solutions of the linear wave equation on more general compact manifolds $\mathcal{M}$ since the linear flow instantly kills the localization. It is worth trying to apply the techniques in this paper to prove the global existence of solutions with data in $H^{\frac{1}{2}} (\mathcal{M}) \times H^{-\frac{1}{2}} (\mathcal{M})$, which would show that the weak turbulence property also holds on $\mathcal{M}$.
\end{rem}

\section{Notation}
\label{Sec:Notation}
Recall that if f is smooth then we have the inversion formula

\begin{equation}
f(x) = \sum_{n \in \mathbb{Z}^{2}} \hat{f}(n) e^{i 2 \pi  n \cdot x}
\nonumber
\end{equation}
Here the $n$\,s denote the \textit{modes}. \\

Let $J:=[a,b]$ be an interval. We say that $u := \Box^{-1}F$ on $J$ if $u$ satisfies on $J$

\begin{equation}
\partial_{tt} u - \triangle u = F \cdot
\nonumber
\end{equation}
We decompose $u$ into its linear part $u_{l}^{J}$ and its nonlinear part $u_{nl}^{J}$ starting from $a$, i.e
$u(t) = u_l^{J}(t) + u_{nl}^J (t)$ with

\begin{equation}
\begin{array}{l}
u_{l}^{J}(t) = \cos{((t-a) D)} u(a) + \frac{\sin{((t-a)D)}}{D} \partial_{t} u(a), \; \text{and}  \\
\\
u_{nl}^{J}(t) = - \int_{a}^{t} \frac{\sin{((t-t')D)}}{D} ( u^{3}(t') ) \, dt' \cdot
\end{array}
\nonumber
\end{equation}

\vspace{5mm}

Now we recall the Paley-Littlewood decomposition.\\
Let $\phi(\xi)$ be a real, radial, nonincreasing function that is equal to $1$ on the unit ball  $\{ \xi \in \mathbb{R}^{2}: |\xi| \leq 1 \}$
and that is supported on $\{ \xi \in \mathbb{R}^{2}: | \xi | \leq 2 \}$. Let $\psi$ denote the function

\begin{equation}
\begin{array}{ll}
\psi(\xi) & := \phi(\xi) - \phi(2 \xi).
\end{array}
\nonumber
\end{equation}
Let $2^{\mathbb{N}}:= \left\{ 2^{k}, \; k \in \mathbb{N} \right\}$. We define the Paley-Littlewood pieces in the following fashion:

\begin{equation}
\begin{array}{ll}
\widehat{P_{M} f}(\xi) & := \psi \left( \frac{\xi}{M} \right) \hat{f}(\xi) \; if \; M \in 2^{\mathbb{N}}, \\
\widehat{P_{0} f}(\xi) & :=  \phi (\xi) \hat{f}(\xi), \\
\widehat{P_{\leq M} f}(\xi) & := \phi \left( \frac{\xi}{M} \right) \hat{f}(\xi),   \\
\widehat{P_{> M} f} (\xi) & := \widehat{f}(\xi) - \widehat{P_{ \leq M} f}(\xi) \; \text{and} \\
P_{ M_2 < \cdot < M_1} f & := P_{< M_1} f - P_{ \leq M_2} f \; if \; (M_1,M_2) \in (2^{\mathbb{N}})^{2} \cdot
\end{array}
\nonumber
\end{equation}
Recall that

\begin{equation}
\begin{array}{ll}
f & = P_{0} f + \sum_{M \in 2^{\mathbb{N}}} P_{M} f \\
P_{\leq M} f & = P_{0} f  + \sum_{Q \in 2^{\mathbb{N}}: Q \leq M} P_{Q} f, \; \text{and} \\
P_{> M} f &  = \sum_{Q \in 2^{\mathbb{N}}: Q > M} P_{Q} f \cdot
\end{array}
\label{Eqn:DecompPaley}
\end{equation}
We also define the \textit{low frequency part} ( resp. the \textit{high frequency part} )
$P_{\lesssim N} f$  ( resp. $P_{\gtrsim N} f$ ):

\begin{equation}
\begin{array}{l}
P_{\lesssim N} f := P_{\leq 128N} f, \\
P_{\gtrsim N} f := P_{ > 128 N} f,
\end{array}
\nonumber
\end{equation}
so that

\begin{equation}
f = P_{\lesssim N} f + P_{\gtrsim N} f \cdot
\nonumber
\end{equation}
Given $ 1 \leq p \leq \infty$ and $M \in 2^{\mathbb{N}} \cup \{ 0 \}$, we recall the following inequality

\begin{equation}
\| P_{M} f \|_{L^{p}(\mathbb{T}^{2})} \,  \lesssim  \| f \|_{L^{p}(\mathbb{T}^{2})} \cdot
\nonumber
\end{equation}
Given $(m,q) \in  \mathbb{R} \times [p, \infty]$, we recall the Bernstein-type inequalities \footnote{More generally
the first estimate holds for a symbol that behaves like $ \langle \xi \rangle^{m}$ or even better, in the sense that it and its derivatives
have similar or better bounds on $ \langle \xi \rangle \sim \langle M \rangle$. Notation: $P_{\leq 0} := P_{0}$. }

\begin{equation}
\begin{array}{ll}
\| P_{M} \langle D \rangle^{m} f \|_{L^{p}(\mathbb{T}^{2})} \sim \langle M \rangle^{m} \| P_{M} f \|_{L^{p}(\mathbb{T}^{2})}, \; \text{and}  \\
\| P_{\leq M} f \|_{L^{q}(\mathbb{T}^{2})} \lesssim \langle M \rangle^{2 \left( \frac{1}{p} - \frac{1}{q} \right)} \| P_{\leq M} f \|_{L^{p}(\mathbb{T}^{2})} \cdot
\end{array}
\nonumber
\end{equation}
We also define $\tilde{\phi}(\xi)$ to be a smooth, real, radial, nonincreasing function that is equal to $1$ on $\{ \xi \in \mathbb{R}^{2}: |\xi| \leq 2 \}$
and that is supported on $\{ \xi \in \mathbb{R}^{2}: | \xi | \leq 4 \}$. Let $\tilde{\psi}(\xi)$ be a smooth, real, radial, nonincreasing function that is
equal to $1$ on $\{ \xi \in \mathbb{R}^{2}:  \frac{1}{2} \leq |\xi| \leq 2 \}$ and that is supported on $\{ \xi \in \mathbb{R}^{2}: \frac{1}{4} \leq
|\xi| \leq 4\}$. \\
\\
Next we recall the Strichartz estimates (see e.g \cite{keeltao,linsog,ginebvelo}) \footnote{
In these papers these estimates are proved on $\mathbb{R}^{2}$. However it is well-known that by truncating
smoothly the periodic initial data up to (say) $[-4 ,4]^{2}$, by considering the solution of (\ref{Eqn:DefocCubWave})
on $\mathbb{R}^{2}$ with these new initial data and by using finite speed of propagation, these estimates also hold
on $\mathbb{T}^{2}$. More generally these estimates hold on compact manifolds without boundary: see \cite{kap}.}
\begin{prop}
Let $J:=[a,b] \subset [0,1]$. Let $u:= \Box^{-1}F$ on $J$. Let $m \in \left\{ \frac{1}{4}, \frac{3}{8} \right\}$. Then we have \footnote{Here $\langle D \rangle$ denotes the
multiplier defined in the Fourier domain by $\widehat{ \langle D \rangle f } (\xi) :=  \langle \xi \rangle \hat{f}(\xi)$}

\begin{itemize}

\item the Strichartz estimates with no gain of derivative

\begin{align}
\| (u,  \partial_{t} \langle D \rangle ^{-1} u) \|_{L_{t}^{q} L_{x}^{r}(J)} \lesssim
\| (u_{0}, u_{1}) \|_{H^{m} (\mathbb{T}^{2}) \times H^{m-1} (\mathbb{T}^{2})}
+ \| F \|_{L_{t}^{\tilde{q}^{'}} L_{x}^{\tilde{r}^{'}} (J) }
\label{Eqn:StrichWave}
\end{align}
\item the Strichartz estimates with gain of derivative

\begin{equation}
\| ( u, \partial_{t} \langle D \rangle ^{-1} u) \|_{L_{t}^{q} L_{x}^{r}(J)} \lesssim \| (u_{0},u_{1}) \|_{H^{m}(\mathbb{T}^{2}) \times H^{m-1}(\mathbb{T}^{2})}
+ \| \langle D \rangle^{m-1} F \|_{L_{t}^{1} L_{x}^{2} (J)}
\label{Eqn:StrichWaveGain}
\end{equation}
under the conditions $(q,r) \in \mathcal{W}_{m}$, $(\tilde{q},\tilde{r}) \in \mathcal{W}'_{m}$. Here

\begin{equation}
\begin{array}{l}
\mathcal{W}_{m} := \left\{ (x,y) \in \mathcal{W}, \, \frac{1}{x} + \frac{2}{y} = 1 - m \right\}, \\
\mathcal{W}'_{m} := \left\{ (x,y) \in \mathcal{W}', \, \frac{1}{x} + \frac{2}{y} - 2 = 1 - m \right\},
\end{array}
\nonumber
\end{equation}
with

\begin{equation}
\begin{array}{l}
\mathcal{W}:= \left\{ (x,y) \in [2, \infty] \times [2, \infty], \,  (x,y) \neq (\infty,\infty), \, \frac{1}{x} + \frac{1}{2y} \leq  \frac{1}{4} \right\},
\; \text{and}
 \\
\mathcal{W}' := \left\{ (x',y'): \exists (x,y) \in \mathcal{W} \, s.t \, \frac{1}{x} + \frac{1}{x'} = 1 \, and \,
\frac{1}{y} + \frac{1}{y'} =1 \right\}.
\end{array}
\nonumber
\end{equation}

\end{itemize}

\end{prop}
Given $w$ a space-time function we define $Z_{m,s}(J,w)$ to be the following number

\begin{equation}
Z_{m,s}(J,w) :=
\left\| \left( \langle D \rangle^{1-m} I w  , \partial_{t}  \langle D \rangle^{-m} I w \right) \right\|_{L_{t}^{\frac{3}{m}} L_{x}^{\frac{6}{3-4m}}(J)} \cdot
\nonumber
\end{equation}
Observe that $\left( \frac{3}{m}, \frac{6}{3-4m} \right) \in \mathcal{W}_{m}$.

\section{General strategy}
\label{Sec:Strategy}

In this section, we explain the general strategy to prove global existence for data
$(u_{0},u_{1}) \in H^{s}(\mathbb{T}^{2}) \times H^{s-1} (\mathbb{T}^{2})$, $s<1$. \\
Assume towards a contradiction that $T_{+} < \infty$. Then we necessarily have

\begin{equation}
\begin{array}{l}
\lim_{t \rightarrow T_{+}^{-}} \left\| ( u(t),\partial_{t} u(t) ) \right\|_{H^{s} (\mathbb{T}^{2}) \times H^{s-1} (\mathbb{T}^{2})}  = \infty.
\end{array}
\label{Eqn:Hsblowupcrit}
\end{equation}
Our goal is then to prove that for some $\epsilon < 1$ fixed and small enough

\begin{equation}
\begin{array}{l}
\sup_{t \in [T_{+} - \epsilon, T_{+})} \| ( u(t), \partial_{t} u(t) ) \|_{H^{s} (\mathbb{T}^{2}) \times H^{s-1} (\mathbb{T}^{2})} < \infty,
\end{array}
\nonumber
\end{equation}
which contradicts (\ref{Eqn:Hsblowupcrit}). A similar argument shows that $T_{-} = - \infty$. \\
Translating in time if necessary we may assume that $T_{+} = \epsilon$; therefore we are reduced to show that

\begin{equation}
\sup_{t \in [0, \epsilon)} \| ( u(t), \partial_{t} u(t) ) \|_{H^{s} (\mathbb{T}^{2}) \times H^{s-1} (\mathbb{T}^{2})} < \infty \cdot
\label{Eqn:NrjBdImx}
\end{equation}
We first prove that there exists $C:= C \left( \| (u_{0}, u_{1}) \|_{H^{s}(\mathbb{T}^{2}) \times H^{s-1}(\mathbb{T}^{2})} \right)$ such that

\begin{equation}
E(Iu_{0}) \leq C N^{2(1-s)}:
\label{Eqn:StratEstMolNrjInit}
\end{equation}
see Subsection \ref{Section:MolNrjInit}. \\
Next we prove that for $s > s_{0}$ (with $s_{0}$ and $N \gg 1$ to be determined)

\begin{equation}
\begin{array}{ll}
\sup_{t \in [0, \epsilon)} E(Iu(t)) & \leq 8 C N^{2(1-s)} :
\end{array}
\label{Eqn:PostBdMolNrj}
\end{equation}
this yields global existence for $s > s_{0}$, in view of Proposition \ref{Prop:MolNrjfinal}. \\
Assume towards a contradiction that (\ref{Eqn:PostBdMolNrj}) does not hold, then this means that there exists $\epsilon_{0} \in [ 0, \epsilon)$
such that

\begin{equation}
\begin{array}{ll}
E(Iu(\epsilon_{0})) & = 4 C N^{2(1-s)},
\end{array}
\label{Eqn:MolNrjEpsilonZero}
\end{equation}
and

\begin{equation}
\begin{array}{ll}
\sup_{t \in [0, \epsilon_{0}]} E(Iu(t)) & \leq 4 C N^{2(1-s)}.
\end{array}
\label{Eqn:AprBdMolNrj}
\end{equation}
Our goal is then to prove that (\ref{Eqn:MolNrjEpsilonZero}) cannot hold. There are three steps:

\begin{enumerate}

\item Local boundedness: control of some norms on an interval $J \subset [0,\epsilon_{0}]$ ``small'': this is done by using
(\ref{Eqn:AprBdMolNrj}), (\ref{Eqn:StrichWave}), and (\ref{Eqn:StrichWaveGain}).
\item Local variation of $E(Iu)$: estimate on $J$ ``small'' (in a sense to be determined) of the variation of the mollified energy by using the
local bounds of these norms
\item Total variation of $E(Iu)$: estimate on $[0, \epsilon_{0}]$ of the
variation of the mollified energy by partitionning $[0,\epsilon_{0}]$ into subintervals $J$ of same size, except maybe the last one. One proves
that

\begin{equation}
\sup_{t \in [0,\epsilon_{0}]} |E(Iu(t)) - E(Iu_{0})| \leq 2 C N^{2(1-s)} \cdot
\label{Eqn:EstNrjGoal}
\end{equation}

\end{enumerate}
This contradicts (\ref{Eqn:MolNrjEpsilonZero}), in view of (\ref{Eqn:StratEstMolNrjInit}).

\section{The kinetic bounds, the potential bound and the dispersive  bounds}
\label{Sec:Bounds}

Throughout this paper we use three types of bounds:

\begin{itemize}

\item \textit{the kinetic bounds}: notice from (\ref{Eqn:AprBdMolNrj}) that the following
bound of the kinetic part of the mollified energy holds, i.e

\begin{equation}
\begin{array}{ll}
\sup_{t \in [0, \epsilon_{0}]} \left\| \left( \nabla I u(t) ,\partial_{t} I u(t) \right) \right\|_{L_{t}^{\infty} L_{x}^{2}([0,\epsilon_{0}]) \times L_{t}^{\infty} L_{x}^{2}([0,\epsilon_{0}])} & \lesssim N^{1-s}.
\end{array}
\label{Eqn:KinetBound}
\end{equation}
The kinetic bounds are the bounds of the type $L_{t}^{\infty} L_{x}^{r} (J)$ (with $ \infty > r \geq 2 $) that can be derived from
(\ref{Eqn:KinetBound}) by using universal estimates such as Sobolev inequalities, Bernstein-type inequalities, or fundamental theorem of calculus. For
example, from

\begin{equation}
\begin{array}{ll}
\sup_{t \in [0, \epsilon_{0}]} \| I u(t) \|_{L^{2}(\mathbb{T}^{2})} \lesssim \| u_{0} \|_{L^{2} (\mathbb{T}^{2})} +
\sup_{t \in [0, \epsilon_{0}]} \| \partial_{t} I u(t) \|_{L^{2}(\mathbb{T}^{2})},
\end{array}
\label{Eqn:KinetMassBd0}
\end{equation}
we see that

\begin{equation}
\begin{array}{ll}
\left\| \left( \langle D \rangle I u ,\partial_{t} I u \right) \right\|_{L_{t}^{\infty} L_{x}^{2}([0,\epsilon_{0}]) \times L_{t}^{\infty} L_{x}^{2}([0,\epsilon_{0}])} & \lesssim N^{1-s}.
\end{array}
\label{Eqn:KinetMassBd1}
\end{equation}
We also have

\begin{equation}
\begin{array}{ll}
\left\| \langle D \rangle^{1-m} I u \right\|_{L_{t}^{\infty} L_{x}^{\frac{2}{1-m}}([0, \epsilon_{0}])} & \lesssim
\| \langle D \rangle I u \|_{L_{t}^{\infty} L_{x}^{2}([0, \epsilon_{0}])}:
\end{array}
\label{Eqn:KinetMassBd2}
\end{equation}
Hence (\ref{Eqn:KinetMassBd0}), (\ref{Eqn:KinetMassBd1}), and (\ref{Eqn:KinetMassBd2}) are kinetic bounds.

\item \textit{the potential bound}: we see from (\ref{Eqn:AprBdMolNrj}) that the following potential bound holds:

\begin{equation}
\| I u \|_{L_{t}^{\infty} L_{x}^{4} ([0,\epsilon_{0}])} \lesssim N^{\frac{1-s}{2}}.
\label{Eqn:BoundPot}
\end{equation}
Notice that this bound is stronger than the one that would be found by using only (\ref{Eqn:KinetBound}). Indeed if we were to
bound $ \| I u \|_{L_{t}^{\infty} L_{x}^{4} ([0,\epsilon_{0}])} $ by using only (\ref{Eqn:KinetBound}), we would find
\begin{equation}
\begin{array}{l}
\| I u \|_{L_{t}^{\infty} L_{x}^{4} ([0,\epsilon_{0}])}  \lesssim \| \langle D \rangle I u \|_{L_{t}^{\infty} L_{x}^{2}([0,\epsilon_{0}])}
\lesssim N^{1-s}:
\end{array}
\label{Eqn:BoundPotNrj}
\end{equation}
this is a worse estimate than (\ref{Eqn:BoundPot}). Therefore this estimate is useful (in fact crucial), in particular
when we deal with multilinear expressions where the low frequency part of $u$ appears : see Subsection \ref{Subsec:LocalBd} (resp. Subsection \ref{Subsec:EstLinNonlin}) and in particular Remark \ref{Rem:LowNNN} (resp. Remark \ref{Rem:InterpAgain}).

\item \textit{the dispersive bounds}: the bounds of the type $L_{t}^{q} L_{x}^{r}$ (with $q < \infty$ ) that are derived from
the Strichartz estimates (\ref{Eqn:StrichWave}) and (\ref{Eqn:StrichWaveGain}), and (\ref{Eqn:KinetBound}). We shall see in Proposition \ref{Prop:LocalBoundedness} that if $J \subset [0, \epsilon_{0}]$ is small enough then

\begin{equation}
\begin{array}{ll}
Z_{m,s}(J,u) & \lesssim N^{1-s},
\end{array}
\nonumber
\end{equation}
by using (\ref{Eqn:StrichWave}) and (\ref{Eqn:KinetBound}). Hence $Z_{m,s}(J,u)$ is a dispersive bound of $u$. We shall use dispersive bounds when we deal with multilinear expressions where the high frequency part of a function $f$ appears: see Subsection \ref{Section:Gwp12} and in particular Remark \ref{Rem:LowNNHighN} and Remark \ref{rem:Y1diff}.

\end{itemize}

\section{Preliminaries}
In this section we prove some preliminary results.

\subsection{Proposition: estimate of mollified energy at time $0$ }
\label{Section:MolNrjInit}
In the following proposition, we estimate the mollified energy at time $0$. We have

\begin{prop}
We have (with $C:= C(\|(u_{0},u_{1}) \|_{H^{s}(\mathbb{T}^{2}) \times H^{s-1}(\mathbb{T}^{2})}) \gg 1$)
\begin{align}
E(Iu_{0}) \leq C N^{2(1-s)} \cdot
\label{Eqn:EstMolNrjInit}
\end{align}
\label{Prop:EstMolNrjInit}
\end{prop}

\begin{proof}

From Plancherel theorem

\begin{align*}
\| \nabla I u_{0} \|_{L^{2} (\mathbb{T}^{2})}^{2} & \lesssim \sum_{ n \in \mathbb{Z}^{2}: |n| \lesssim N}  |n|^{2} | \widehat{u_{0}}(n) |^{2}
+ \sum_{ n \in \mathbb{Z}^{2}: |n| \gtrsim N} \frac{N^{2(1-s)}}{|n|^{2(1-s)}} |n|^{2}  | \widehat{u_{0}}(n) |^{2}  \\
& \lesssim N^{2(1-s)} \cdot
\end{align*}
Similarly

\begin{align*}
\| I u_{1} \|_{L^{2} (\mathbb{T}^{2})}^{2} \lesssim N^{2(1-s)} \cdot
\end{align*}
From Bernstein-type inequalities we easily get

\begin{align*}
\| I u_{0} \|^{4}_{L^{4} (\mathbb{T}^{2})} & \lesssim   \| P_{\lesssim N} I u_{0} \|^{4}_{L^{4} (\mathbb{T}^{2})} +
\| P_{\gtrsim N} I u_{0} \|^{4}_{L^{4} (\mathbb{T}^{2}) }  \nonumber \\
& \lesssim \max{(N^{2-4s},1)} \nonumber \\
& \lesssim N^{2(1-s)} \cdot
\end{align*}

\end{proof}

\subsection{Proposition: estimate of Sobolev norms}
\label{Section:MolNrjFinal}

In the following proposition, we prove that a bound of the mollified energy automatically yields a bound of
the Sobolev norms of the solution. We have

\begin{prop}
Let $ T \in [0,\epsilon)$. Then

\begin{align}
\left\| (u(T), \partial_{t} u(T)) \right\|^{2}_{H^{s}(\mathbb{T}^{2}) \times H^{s-1}(\mathbb{T}^{2})}  \lesssim
\| (u_{0},u_{1}) \|^{2}_{H^{s}(\mathbb{T}^{2}) \times H^{s-1}(\mathbb{T}^{2})} + \sup_{t \in [0,T]} E(Iu(t))
\end{align}
\label{Prop:MolNrjfinal}
\end{prop}

\begin{proof}

From the fundamental theorem of calculus

\begin{align*}
\| P_{0} u(T) \|_{H^{s} (\mathbb{T}^{2})} \lesssim \| I u(T) \|_{L^{2} (\mathbb{T}^{2})}
\lesssim \| u_0 \|_{L^{2}(\mathbb{T}^{2})} + \sup_{t \in [0,T]} \| \partial_{t} I u (t) \|_{L^{2} (\mathbb{T}^{2})}
\end{align*}
From Plancherel theorem

\begin{align*}
\| P_{> 0} u(T) \|^{2}_{H^{s}(\mathbb{T}^{2})} & \lesssim  \| P_{ 0 < \cdot \lesssim N} u(T) \|^{2}_{H^{s} (\mathbb{T}^{2})} +
\| P_{ \gtrsim N} u(T)  \|^{2}_{H^{s}(\mathbb{T}^{2})} \\
& \lesssim \sum_{n \in \mathbb{Z}^{2}: |n| \lesssim N} |n|^{2} |\hat{u}(T,n)|^{2}  +
\sum_{n \in \mathbb{Z}^{2}: |n| \gtrsim N} |n|^{2} \frac{N^{2(1-s)}}{|n|^{2(1-s)}} | \hat{u}(T,n)|^{2}   \\
& \lesssim \| \nabla I u(T) \|^{2}_{L^{2}(\mathbb{T}^{2})} \cdot
\end{align*}
Dividing into modes $|n| \lesssim N$ and $|n| \gtrsim N$ we get in a similar fashion

\begin{align*}
\| \partial_{t} u(T) \|_{H^{s-1}(\mathbb{T}^{2})}^{2} \lesssim
\| \partial_{t} I u(T) \|^{2}_{L^{2} (\mathbb{T}^{2})} \cdot
\end{align*}

\end{proof}









\section{Global existence for $s> \frac{4}{9}$}
\label{Section:Gwp12}

In this section, we prove the global existence of a solution $u$ of (\ref{Eqn:DefocCubWave}) with data in $H^{s}(\mathbb{T}^{2})
\times H^{s-1}(\mathbb{T}^{2})$, $s > \frac{4}{9}$.

\subsection{Local boundedness}
\label{Subsec:LocalBd}
In this subsection, we prove that we get some dispersive bounds of the solution on an
interval $J$ satisfying a smallness condition (namely the condition (\ref{Eqn:SizeJ})). In Subsection \ref{Subsec:LocalVarEIu}, we will use these bounds to estimate the variation of the mollified energy
on $J$. In Subsection \ref{Subsec:TotalVarEIu}, we will estimate the variation of the mollified energy on the whole
interval $[0,\epsilon_{0}]$ by partionning it into intervals $J$ satisfying (\ref{Eqn:SizeJ}), and an iteration process. \\
Our goal is to prove these bounds on an interval $J$ as large as possible: this allows to reduce the number of intervals
$J$ covering $[0,\epsilon_{0}]$, which eventually yields a better estimate of the variation of the mollified energy on $[0, \epsilon_{0}]$ and
therefore prove the global existence for rougher data. \\
We have the following:

\begin{prop}
Let $J:=[a,b] \subset [0,\epsilon_{0}]$. There exists $ 0 < c \ll 1$ such that if
\begin{align}
|J| \leq c N^{s-1},
\label{Eqn:SizeJ}
\end{align}
then

\begin{equation}
\begin{array}{l}
m \in \left\{ \frac{1}{4}, \frac{3}{8} \right\}: Z_{m,s}(J,u) \lesssim N^{1-s}.
\end{array}
\label{Eqn:EstZms}
\end{equation}

\label{Prop:LocalBoundedness}
\end{prop}

\begin{proof}

From (\ref{Eqn:StrichWave}) and (\ref{Eqn:KinetMassBd1}) we get

\begin{align*}
Z_{m,s}(J,u) & \lesssim  \left\|  \langle D \rangle I u(a),\partial_{t} I u(a)) \right\|_{L^{2} (\mathbb{T}^{2}) \times L^{2} (\mathbb{T}^{2})} +
\| \langle D \rangle^{1-m} I (P_{\lesssim N} u)^{3}  \|_{L_{t}^{1} L_{x}^{\frac{2}{2-m}} (J)} +  \\
& \| \langle D \rangle^{1-m} I (P_{\lesssim N} u P_{\lesssim N} u P_{ \gtrsim N} u) \|_{L_{t}^{1} L_{x}^{\frac{2}{2-m}} (J)} + \\
& \| \langle D \rangle^{1-m} I (P_{\lesssim N} u P_{ \gtrsim N} u P_{ \gtrsim N} u)  \|_{L_{t}^{\frac{3}{2+m}} L_{x}^{\frac{6}{7-4m}}(J)} + \\
& \| \langle D \rangle ^{1-m} I (P_{\gtrsim N} u)^{3}  \|_{L_{t}^{\frac{3}{2+m}} L_{x}^{\frac{6}{7-4m}}(J)} \\
& \lesssim N^{1-s} + Y_{1} + Y_{2} + Y_{3} + Y_{4} \cdot
\end{align*}
We have

\begin{equation}
\begin{array}{ll}
Y_{1} & \lesssim  |J| \, \| \langle D \rangle^{1-m} I  u \|_{L_{t}^{\infty} L_{x}^{\frac{2}{1-m}} (J)}  \| P_{\lesssim  N} u \|^{2}_{L_{t}^{\infty} L_{x}^{4}(J)} \\
& \lesssim |J|  \| \langle D \rangle I u \|_{L_{t}^{\infty} L_{x}^{2} (J)} \| I u \|^{2}_{L_{t}^{\infty} L_{x}^{4}(J)} \\
& \lesssim |J| N^{2(1-s)},
\end{array}
\label{Eqn:EstLowmodesNNN}
\end{equation}
using at the last line (\ref{Eqn:BoundPot}).

\begin{rem}
We make two comments regarding the estimate above:

\begin{itemize}

\item since we work on intervals $J$ with size smaller than one, it is better to create powers of $|J|$ as large as possible by H\"older in time since
we want to prove these bounds on $J$ as large as possible. Therefore we place the nonlinearity in the space $L_{t}^{q} L_{x}^{r}(J)$ with the
smallest index $q$ such that $(q,r) \in \mathcal{W}_{m}^{'}$ and we control it by using the kinetic bounds (and the potential one) of the solution.

\item we interpolate with the potential bound (\ref{Eqn:BoundPot}) of the mollified energy since $Y_{1}$ is made of the low frequency part of $u$: this
is consistent with Section \ref{Sec:Bounds}. So we also use the global information of the
solution: this allows to lower the power of $N$ in the estimate of $Y_{1}$, which ultimately enlarges the size of $J$ such that
(\ref{Eqn:EstZms}) holds.

\end{itemize}
\label{Rem:LowNNN}
\end{rem}
We have

\begin{equation}
\begin{array}{l}
\| \langle D \rangle^{1-m} I (P_{\lesssim N} u P_{\lesssim N} u P_{\gtrsim N} u) \|_{L_{t}^{1} L_{x}^{\frac{2}{2-m}} (J)}  \\
\lesssim
\left(
\begin{array}{l}
|J| ^{\frac{11}{12}}\| \langle D \rangle^{1-m} I u \|_{L_{t}^{\infty} L_{x}^{\frac{2}{1-m}} (J)} \| P_{ \lesssim N} u \|_{L_{t}^{\infty} L_{x}^{6}(J)}
\| P_{ \gtrsim N} u \|_{L_{t}^{12} L_{x}^{3} (J)}  \\
+ |J| \| \langle D \rangle^{1-m} I u \|_{L_{t}^{\infty} L_{x}^{\frac{2}{1-m}} (J)}  \| P_{ \lesssim N} u \|^{2}_{L_{t}^{\infty} L_{x}^{4}(J)}
\end{array}
\right) \\
\lesssim |J|^{\frac{11}{12}}  X_{1} + |J| X_{2} .
\end{array}
\label{Eqn:EstLowNNHighN}
\end{equation}
By the previous computations $X_{2} \lesssim N^{2(1-s)} $. We have \\

\begin{equation}
\begin{array}{ll}
X_{1} & \lesssim  \| \langle D \rangle  I u \|_{L_{t}^{\infty} L_{x}^{2}(J)} N^{\frac{2(1-s)-}{3}}
\frac{\| \langle D \rangle^{1 - \frac{1}{4}}  I u \|_{L_{t}^{12} L_{x}^{3}(J)}}{N^{\frac{3}{4}}} \\
& \lesssim  \frac{N^{\frac{8(1-s)}{3}}}{N^{\frac{3}{4}}},
\end{array}
\label{Eqn:EstllonNNHighN}
\end{equation}
since (using again (\ref{Eqn:BoundPot}))

\begin{equation}
\begin{array}{ll}
\| P_{ \lesssim N} u \|_{L_{t}^{\infty} L_{x}^{6}(J)}  & \lesssim   \| I u \|^{\frac{1}{3}+}_{L_{t}^{\infty} L_{x}^{\infty-} (J)}
\| I u \|^{\frac{2}{3}-}_{L_{t}^{\infty} L_{x}^{4}(J)} \\
& \lesssim  N^{\frac{2(1-s)}{3}-}.
\end{array}
\nonumber
\end{equation}

\begin{rem}
We make two comments regarding the estimate (\ref{Eqn:EstLowNNHighN}). \\

\begin{itemize}

\item Notice that for the high frequency part of $u$ we use a dispersive bound. The dispersive bound creates a power $N^{\frac{3}{4}}$: this comes
from the fact that we work with data that have some regularity (the data is in $H^{s} (\mathbb{T}^{2}) \times H^{s-1} (\mathbb{T}^{2})$, $s > \frac{3}{8}$).
If we had chosen kinetic bounds to estimate \\
$\| \langle D \rangle^{1-m} I (P_{\lesssim N} u P_{\lesssim N} u P_{\gtrsim N} u) \|_{L_{t}^{1} L_{x}^{\frac{2}{2-m}} (J)}$, we would have found by using the same scheme as (\ref{Eqn:EstLowmodesNNN})

\begin{equation}
\begin{array}{l}
\| \langle D \rangle^{1-m} I (( P_{\lesssim N} u)^{2} P_{ \gtrsim N} u) \|_{L_{t}^{1} L_{x}^{\frac{2}{2-m}}(J)} \\
\lesssim |J|
\| \langle D \rangle^{1-m} I u \|_{L_{t}^{\infty} L_{x}^{\frac{2}{1-m}}(J)} \| P_{\lesssim N} u \|_{L_{t}^{\infty} L_{x}^{6}(J)}
\| P_{\gtrsim N} u \|_{L_{t}^{\infty} L_{x}^{3}(J)} + |J| X_{2} \\
\lesssim |J| N^{\frac{5(1-s)}{3}} \frac{ \| \langle D \rangle^{1- \frac{1}{3}} I u \|_{L_{t}^{\infty} L_{x}^{3}(J)}} {N^{\frac{1}{3}}} + |J| X_{2} \\
 \lesssim |J| N^{\frac{5(1-s)}{3}} \frac{ \| \langle  D \rangle I u \|_{L_{t}^{\infty} L_{x}^{2}(J)}} {N^{\frac{1}{3}}}
+ |J| X_{2} \\
 \lesssim |J| \frac{N^{\frac{8(1-s)}{3}}}{N^{\frac{1}{3}}} + |J| X_{2}:
\end{array}
\nonumber
\end{equation}
therefore it is a worse estimate than (\ref{Eqn:EstllonNNHighN}) for $|J| \sim N^{s-1}$ (this is the maximal size allowed by (\ref{Eqn:SizeJ})). In fact, this dispersive bound yields better results for regularity purposes, since we control $\| \langle D \rangle^{1-\frac{1}{4}} I u \|_{L_{t}^{12} L_{x}^{3}(J)}$ instead of $\| \langle D \rangle^{1-\frac{1}{3}} I u \|_{L_{t}^{\infty} L_{x}^{3}(J)}$. But there is a price to pay: we lose integrability in time. We are in $L_{t}^{12}$ and not in $L_{t}^{\infty}$, so we create a smaller power of $|J|$ by H\"older in time, which is a disadvantage in view of the first comment of
Remark \ref{Rem:LowNNN}. The computations showed for this example that the loss of integrability effect of the dispersive bound is weaker than the gain of
regularity effect. In fact, this observation is rather general and we shall use it until the end of this manuscript: we let the reader check
that it is better to use dispersive bounds for the high frequency part than the kinetic bounds. This is consistent
with Section \ref{Sec:Bounds}.

\item Notice that again we interpolate with the potential bound (\ref{Eqn:BoundPot}) of the mollified
energy for the low frequency part of $u$. This is again consistent with Section
\ref{Sec:Bounds} (see also the second point of Remark \ref{Rem:LowNNN}).

\end{itemize}
\label{Rem:LowNNHighN}
\end{rem}

We also have

\begin{equation}
\begin{array}{l}
\| \langle D \rangle^{1-m} I  (P_{\lesssim N} u P_{ \gtrsim N} u P_{ \gtrsim N} u)  \|_{L_{t}^{\frac{3}{2+m}} L_{x}^{\frac{6}{7-4m}}(J)} \\
\lesssim
\|\langle D \rangle^{1-m} I u \|_{L_{t}^{\frac{3}{m}} L_{x}^{\frac{6}{3-4m}} (J)} \left( |J|^{\frac{7}{12}}
\| P _{ \gtrsim N} u \|_{ L_{t}^{12} L_{x}^{3} (J)} \| P_{\lesssim N} u  \|_{L_{t}^{\infty} L_{x}^{3} (J)}
+ |J|^{\frac{1}{2}} \| P _{ \gtrsim N} u \|^{2}_{ L_{t}^{12} L_{x}^{3} (J)} \right)
 \\
\lesssim  Z_{m,s} (J,u)
\left( |J|^{\frac{7}{12}}
\frac{ \| \langle D \rangle^{1- \frac{1}{4}} I u \|_{L_{t}^{12} L_{x}^{3} (J)} }{N^{\frac{3}{4}}} \| I u \|_{L_{t}^{\infty} L_{x}^{4} (J)}
+ |J|^{\frac{1}{2}} \frac{ \| \langle D \rangle^{1- \frac{1}{4}} I u \|^{2}_{L_{t}^{12} L_{x}^{3} (J)} }{N^{\frac{3}{2}}}
\right) \\
\lesssim   Z_{m,s} (J,u) Z_{\frac{1}{4},s}(J,u)
\left( |J|^{\frac{7}{12}} \frac{N^{\frac{1-s}{2}}}{N^{\frac{3}{4}}} + |J|^{\frac{1}{2}} \frac{Z_{\frac{1}{4},s}(J,u)}{N^{\frac{3}{2}}}  \right).
\end{array}
\nonumber
\end{equation}
We also have

\begin{equation}
\begin{array}{ll}
\| \langle D \rangle^{1-m} I  ((P_{ \gtrsim N} u)^{3} ) \|_{L_{t}^{\frac{3}{2+m}} L_{x}^{\frac{6}{7-4m}}(J)} & \lesssim
\| \langle D \rangle^{1-m} I u \|_{L_{t}^{\frac{3}{m}} L_{x}^{\frac{6}{3-4m}} (J)} \| P_{\gtrsim N} u \|^{2}_{L_{t}^{3} L_{x}^{3} (J)} \\
& \lesssim Z_{m,s}(J,u) |J|^{\frac{1}{2}} \frac{ \| \langle D \rangle^{1- \frac{1}{4}} I u \|^{2}_{L_{t}^{12} L_{x}^{3} (J)} }{N^{\frac{3}{2}}}.
\end{array}
\nonumber
\end{equation}
In view of the above estimates we see that, using twice a continuity argument (first for $m=\frac{1}{4}$ and then, for $m = \frac{3}{8}$)

\begin{align*}
Z_{m,s}(J,u) \lesssim N^{1-s}.
\end{align*}

\begin{rem}
The reader can check that $Y_{1}$ is the most difficult term to bound: this is due to the fact that the other terms contain
the high frequency part of $u$ and, consequently, we can use the dispersive bounds to
our advantage (see the first point of Remark \ref{Rem:LowNNHighN}). Therefore the interpolation with the potential bound in estimating $Y_{1}$
is crucial.
\label{rem:Y1diff}
\end{rem}


\end{proof}

\subsection{Local variation of $E(Iu)$}
\label{Subsec:LocalVarEIu}

In this subsection, we estimate the variation of the mollified energy on $J$ satisfying (\ref{Eqn:SizeJ}).

\begin{prop}
Let $J =[a,b] \subset [0,\epsilon_{0}]$ satisfying (\ref{Eqn:SizeJ}). Let $s > \frac{3}{8}$. The variation
$Var(E(Iu)), J) := E(Iu(b)) - E(Iu(a))$ of $E(Iu)$  is given by

\begin{equation}
\begin{array}{ll}
Var(E(Iu),J) & \lesssim N^{4(1-s)}
  \left( \frac{|J|^{\frac{3}{4}}}{ N^{\frac{3}{2}-} } +
\frac{ |J|^{\frac{1}{2}}}{N^{\frac{5}{4}-}} \right) \cdot
\end{array}
\nonumber
\end{equation}
\label{Prop:EstVarNrj}
\end{prop}

\subsubsection{The proof}

Plugging $I$ into (\ref{Eqn:DefocCubWave}), computing the derivative of $E(Iu)$ and integrating on $J$
we see that

\begin{equation}
\begin{array}{ll}
Var(E(Iu),J)  & = \left| \int_{J} \int_{\mathbb{T}^{2}}
\partial_{t} I u  \left( (Iu)^{3} - I u^{3} \right) \, dx \, dt \right|.
\end{array}
\nonumber
\end{equation}
Hence

\begin{equation}
\begin{array}{l}
Var(E(Iu),J) \\
= \left| \sum \limits_{n_{1} + n_{2} + ... + n_{4} = 0} \int_{J}  \mu(n_{1},n_{2},n_{3},n_{4})
\widehat{\partial_{t} I u}(t,n_{1}) \widehat{Iu} (t,n_{2}) \widehat{Iu}(t,n_{3})
\widehat{Iu}(t,n_{4})  \, dt \right|,
\end{array}
\label{Eqn:FourEIu}
\end{equation}
with

\begin{equation}
\begin{array}{ll}
\mu(n_{1},n_{2},n_{3},n_{4}) & = 1 - \frac{m(n_{1})}{m(n_{2}) m(n_{3}) m(n_{4})}.
\end{array}
\nonumber
\end{equation}
Hence plugging the decomposition (\ref{Eqn:DecompPaley}) we are reduced to estimate

\begin{equation}
\begin{array}{l}
\sum_{ \vec{M} \in (2^{\mathbb{N}} \cup \{ 0 \})^{4}  } X_{\vec{M}},
\end{array}
\nonumber
\end{equation}
with

\begin{equation}
\begin{array}{l}
X_{\vec{M}} := \sum \limits_{n_{1} +... + n_{4} =0} \int_{J}
\mu(n_{1}, n_{2}, n_{3}, n_{4} ) \widehat{\partial_{t} I u_{1}}(t,n_{1}) \widehat{Iu_{2}} (t,n_{2}) \widehat{I u_{3}}(t,n_{3})
\widehat{Iu_{4}}(t,n_{4}) dt,
\end{array}
\nonumber
\end{equation}
$\vec{M}:= (M_{1},...,M_{4})$, and $u_{i} := P_{M_{i}} u$, $ 1 \leq  i \leq 4$.

\textbf{Strategy}: We have

\begin{equation}
\begin{array}{l}
|X_{\vec{M}}|  \lesssim  B(\vec{M}) |J|^{\frac{1}{q_{0}}}
\| \partial_{t}  I  u_{1} \|_{L_{t}^{q_{1}} L_{x}^{r_{1}}(J)}
\| I u_{2} \|_{L_{t}^{q_{2}} L_{x}^{r_{2}} (J)} ...
\| I u_{4} \|_{L_{t}^{q_{4}} L_{x}^{r_{4}} (J)}, \\
with \, B(\vec{M}): \, |\mu(\xi_{1},\xi_{2},\xi_{3},\xi_{4}| \lesssim B \, for \, |\xi_{1}| = |\xi_{2}+ \xi_{3} + \xi_{4}| \sim M_{1},..., |\xi_{4}| \sim M_{4},
\end{array}
\label{Eqn:ConclStrategy}
\end{equation}
with

\begin{equation}
\begin{array}{l}
(r_{1},...,r_{4}) \in (1, \infty)^{4}: \, \frac{1}{r_{1}} + \frac{1}{r_{2}} + \frac{1}{r_{3}} + \frac{1}{r_{4}} =1, \\
(q_{0},...,q_{4}) \in [1, \infty]^{4}: \, \frac{1}{q_{0}} + \frac{1}{q_{1}} + \frac{1}{q_{2}} + \frac{1}{q_{3}} + \frac{1}{q_{4}} =1.
\end{array}.
\nonumber
\end{equation}
The procedure is explained in Subsection \ref{Subsec:Procedure}. \\
By symmetry, we may assume that $M_{2} \geq M_{3} \geq M_{4}$.
Let  $ M_{i}^{*} $ be the dyadic numbers $M_{1}$,...,$M_{4}$ in order, i.e
$M_{1}^{*} \geq M_{2}^{*} \geq M_{3}^{*} \geq M_{4}^{*}$. In view of the convolution constraint, we may assume that
$M_{1}^{*} \sim M^{*}_{2}$ . We may also assume that $M_{1}^{*} \gtrsim N$ since if not the multiplier $\mu$ vanishes.
There are several cases \footnote{it is recommended that the reader ignores at the first reading the $M_{2}^{-}$ that is created to
make the terms summable.}:

\begin{itemize}

\item \textbf{Case} $\mathbf{1}$ : $ M_{1}^{*}= M_{1} $ and $ M_{2}^{*} =M_{2}$

\begin{itemize}

\item \textbf{Case} $\mathbf{1.a}$: $M_{3} \gtrsim N$. Then

\begin{equation}
\begin{array}{ll}
|\mu| & \lesssim  \frac{1}{m(M_{3}) m(M_{4})}.
\end{array}
\label{Eqn:MultCase1a}
\end{equation}
From (\ref{Eqn:KinetMassBd1}) we see that

\begin{equation}
\begin{array}{ll}
\| I u_{4} \|_{L_{t}^{\infty} L_{x}^{\infty-} (J)} & \lesssim \left\| \langle D \rangle I u_{4} \right\|_{L_{t}^{\infty} L_{x}^{2}(J)} \\
& \lesssim N^{1-s} \cdot
\end{array}
\nonumber
\end{equation}
Hence, using also Proposition \ref{Prop:LocalBoundedness}

\begin{equation}
\begin{array}{l}
|X_{\vec{M}}| \lesssim \frac{1}{m(M_{3}) m(M_{4})} |J|^{\frac{3}{4}} \| \partial_{t} I u_{1} \|_{L_{t}^{\infty} L_{x}^{2+} (J)}
\| I u_{2} \|_{L_{t}^{8} L_{x}^{4} (J)} \| I u_{3} \|_{L_{t}^{8} L_{x}^{4} (J)} \| I u_{4} \|_{L_{t}^{\infty} L_{x}^{\infty-} (J)} \\
\lesssim \frac{M_{1}^{+}}{m(M_{3}) m(M_{4})} |J|^{\frac{3}{4}} \frac{1}{M_{2}^{\frac{5}{8}}} \frac{1}{M_{3}^{\frac{5}{8}-}}  Z^{2}_{\frac{3}{8},s}(J,u) N^{2(1-s)} \\
\lesssim   |J|^{\frac{3}{4}} \frac{M_{2}^{-} N^{4(1-s)} }{N^{\frac{5}{4}-}} \cdot
\end{array}
\label{Eqn:EstCase1a}
\end{equation}

\item \textbf{Case} $\mathbf{1.b}$: $M_{3} \lesssim N$.

\begin{equation}
\begin{array}{ll}
|\mu| & \lesssim   \frac{|\nabla m(n_{2})| |n_{3} + n_{4}|}{m(n_{2})}  \\
& \lesssim \frac{M_{3}}{M_{2}} \cdot
\end{array}
\label{Eqn:MultCase1b}
\end{equation}
From this estimate and placing $u_{1}$, ..., $u_{4}$ in the same spaces as in the previous case, we also get

\begin{equation}
\begin{array}{ll}
| X_{\vec{M}} |  & \lesssim |J|^{\frac{3}{4}}  \frac{M_{2}^{-} N^{4(1-s)} }{N^{\frac{5}{4}-}} \cdot
\end{array}
\label{Eqn:EstXZ}
\end{equation}

\end{itemize}

\item \textbf{Case} $\mathbf{2}$ : $ M_{1}^{*}= M_{2} $ and $ M_{2}^{*} =M_{3}$

\begin{equation}
\begin{array}{ll}
|\mu| & \lesssim \frac{m(M_{1})}{(m(M_{2}))^{2} m(M_{4})} \cdot
\end{array}
\label{Eqn:EstMultCase2}
\end{equation}
There are two subcases:

\begin{itemize}

\item \textbf{Case} $\mathbf{2.a}$: $M_{4} \lesssim N$

Using this estimate and placing $u_{1}$, ..., $u_{4}$ in the same spaces as in Case $1.b$:, we also get
(\ref{Eqn:EstXZ}).

\item \textbf{Case} $\mathbf{2.b}$: $M_{4} \gtrsim N$

\begin{equation}
\begin{array}{ll}
|X_{\vec{M}}| & \lesssim \frac{m(M_1)}{(m(M_{2}))^{2} m(M_{4})} |J|^{\frac{1}{2}}  \| \partial_{t}  I u_{1} \|_{L_{t}^{8} L_{x}^{4} (J)}
\| I u_{2} \|_{L_{t}^{8} L_{x}^{4} (J)} \| I u_{3} \|_{L_{t}^{8} L_{x}^{4} (J) } \| I u_{4} \|_{L_{t}^{8} L_{x}^{4} (J)} \\
& \lesssim \frac{m(M_1{})}{(m(M_{2}))^{2} m(M_{4})} |J|^{\frac{1}{2}}  \frac{ M_{1}^{\frac{3}{8}} }{M_{2}^{\frac{5}{8}} M_{3}^{\frac{5}{8}}
M_{4}^{\frac{5}{8}}}
\left\| \partial_{t} \langle D \rangle^{-\frac{3}{8}} I u_{1} \right\|_{L_{t}^{8} L_{x}^{4} (J)}
\left\| \langle D \rangle^{1 -\frac{3}{8}} I u_{2} \right\|_{L_{t}^{8} L_{x}^{4} (J)} \\
& \left\| \langle D \rangle^{1 -\frac{3}{8}} I u_{3} \right\|_{L_{t}^{8} L_{x}^{4} (J)} \left\| \langle D \rangle^{1 -\frac{3}{8}} I u_{4} \right\|_{L_{t}^{8} L_{x}^{4} (J)}  \\
& \lesssim \frac{m(M_{1})}{(m(M_{2}))^{2} m(M_{4})} \frac{ M_{1}^{\frac{3}{8}} }{M_{2}^{\frac{5}{8}} M_{3}^{\frac{5}{8}} M_{4}^{\frac{5}{8}}}
Z_{\frac{3}{8},s} (J,u_{1})
 Z_{\frac{3}{8},s} (J,u_{2})  Z_{\frac{3}{8},s} (J,u_{3}) Z_{\frac{3}{8},s} (J,u_{4}) \\
& \lesssim |J|^{\frac{1}{2}} \frac{M_{2}^{-} N^{4(1-s)}}{ N^{\frac{3}{2}-}} \cdot
\end{array}
\nonumber
\end{equation}

\end{itemize}

\end{itemize}

\begin{rem}

We make two remarks regarding the proof of Proposition \ref{Prop:EstVarNrj}.

\begin{enumerate}

\item Since we work again on intervals $J$ with size $ < 1$, it is better to create large powers of $|J|$ in order to have a better estimate
of $Var(E(Iu),J)$.

\item Notice that on high frequencies \footnote{i.e for frequencies $ \gtrsim N$} we mostly use the dispersive bounds: this is consistent
with the first point of Remark \ref{Rem:LowNNHighN}.

\end{enumerate}

\end{rem}

\subsection{Total variation of $E(Iu)$ and global well-posedness for $ s > \frac{4}{9}$}
\label{Subsec:TotalVarEIu}

Partitioning $[0,\epsilon_{0}]$ into subintervals $J$ of size $ c N^{s-1} $   with $c$ defined in (\ref{Eqn:SizeJ})
(except maybe the last one), we see by iteration that

\begin{align}
\sup_{t \in [0,\epsilon_{0}]} \left| E(Iu(t)) - E(Iu_{0}) \right|  & \lesssim \frac{1}{|J|} \left(
\frac{ |J|^{\frac{3}{4}} N^{4(1-s)}}{N^{\frac{5}{4}-}} +
\frac{ |J|^{\frac{1}{2}} N^{4(1-s)} }{N^{\frac{3}{2}-}} \right) \nonumber \\
& \leq 2 C N^{2(1-s)}, \nonumber
\end{align}
where at the last line we choose $ N \gg 1$ so that (\ref{Eqn:EstNrjGoal}) holds \footnote{such a choice is possible since $s > \frac{4}{9}$}. This proves global well-posedness for $s > \frac{4}{9}$.

\section{Global well-posedness for $ s > \frac{2}{5}$}
\label{Section:Gwpbelow12}

In this section we prove global well-posedness for $s > \frac{2}{5}$. In order to do that,
we implement the use of the potential bound for the low frequency part within an adapted linear-nonlinear decomposition \cite{troy}. \\
\\
\underline{Assumption}: we may assume without loss of generality that $ s < \frac{1}{2}$ since we have already proved global well-posedness
for $s \geq \frac{1}{2}$ in the previous section. \\

\subsection{Strategy}

We recall the strategy of the adapted-linear nonlinear decomposition \cite{troy}. \\
Our goal is to enlarge the size of $J$ for which we prove dispersive bounds such as
(\ref{Eqn:EstZms}): this allows to reduce the number of iterations to estimate the variation of the mollified energy on the whole interval
$[0,\epsilon_{0}]$, which yields eventually a better estimate of the total variation and consequently global existence of solutions of
(\ref{Eqn:DefocCubWave}) for rougher data. \\
Assume for a while that $u$ does not have a nonlinear part, i.e $u = u_{l}^{[0,\epsilon_{0}]}$, then the Strichartz estimates
(\ref{Eqn:StrichWave}) are in fact global bounds. Indeed, plugging $ \langle D \rangle^{1-m} I$ into (\ref{Eqn:StrichWave}), we see, using (\ref{Eqn:AprBdMolNrj}), that

\begin{equation}
Z_{m,s}([0,\epsilon_{0}],u) \lesssim N^{1-s}.
\nonumber
\end{equation}
Assume now that $u$ does not have a linear part, i.e $u=u_{nl}^{[0,\epsilon_{0}]}$. It was observed in \cite{bourgbook} that the nonlinear part is smoother
than $H^{s}$. We measure in Proposition \ref{Prop:EstLinNonlin} this gain of regularity. The gain of regularity can be used to our advantage to estimate the variation of the mollified energy, since the variation is nonzero only if at least
one of the Paley-Littlewood pieces of the solution is supported on high frequencies: see Subsection \ref{Subsec:LocalVarEIu}. Again,
one would like to prove this gain of regularity on an interval $J:=[0,b] \subset [0,\epsilon_{0}]$ as large as possible: this is mostly done by
interpolation with the potential bound and it is explained in the proof of Proposition \ref{Prop:EstLinNonlin}.\\
The assumptions $u=u_{l}^{[0,\epsilon_{0}]}$ and $u=u_{nl}^{[0,\epsilon_{0}]}$ are of course not true and one has to plug
the following decomposition

\begin{equation}
u=u_{l}^{[0,\epsilon_{0}]} + u_{nl}^{[0,\epsilon_{0}]}
\label{Eqn:DecompPrim}
\end{equation}
into the variation. But there is an issue: if we perform the decomposition (\ref{Eqn:DecompPrim}) the gain of regularity can
only be used up to $b$. Instead we perform an adapted linear-nonlinear decomposition: we subdivide the whole interval $[0,\epsilon_{0}]$ into subintervals $J$ of same size to be determined (except maybe the last one), and we plug $u=u_{l}^{J} + u_{nl}^{J}$ into the variation so that we can use the gain of regularity from $0$ to
$\epsilon_{0}$.

\subsection{Estimates for the linear part and the nonlinear part}
\label{Subsec:EstLinNonlin}

We define the number $\alpha$ to be

\begin{equation}
\begin{array}{ll}
\alpha & :=  \max{ \left( \frac{ |J| N^{2(1-s)+} }{N}, |J|^{\frac{1}{2}} \frac{N^{2(1-s)}}{N^{\frac{3}{2}}},
|J|^{\frac{7}{12}} \frac{N^{\frac{3(1-s)}{2}}}{N^{\frac{3}{4}}},
 |J|^{\frac{1}{2}} \frac{N^{2(1-s)} \langle |J| N^{1-s} \rangle^{\frac{1}{8}}}{N^{\frac{3}{2}}}
\right)} \\
& := \max{\left( \alpha_{1},...,\alpha_{4} \right)}.
\end{array}
\nonumber
\end{equation}

\begin{prop}

Let $J:= [a,b] \subset [0,\epsilon_{0}]$. If $k \in \{ l,nl \}$ then let

\begin{equation}
X_{k} := \max{ \left(
\begin{array}{l}
\| P_{\gtrsim N} \langle D \rangle^{1 - \frac{3}{8}} I u_{k}^{J} \|_{L_{t}^{8} L_{x}^{4} (J)},  \| P_{\gtrsim N} \partial_{t}  \langle D \rangle^{-\frac{3}{8}}
I u_{k}^{J} \|_{L_{t}^{8} L_{x}^{4} (J)} \\
\| P_{\gtrsim N}  \langle D \rangle^{1 - \frac{1}{4}} I u_{k}^{J} \|_{L_{t}^{12} L_{x}^{3}(J)},
\| P_{\gtrsim N} \langle D \rangle^{-\frac{1}{4}} \partial_{t}  I u_{k}^{J} \|_{L_{t}^{12} L_{x}^{3}(J)}
\end{array}
 \right) }.
\end{equation}
Then

\begin{align}
X_{l} \lesssim N^{1-s}.
\label{Eqn:Boundalphillin}
\end{align}
Assume furthermore that $ \alpha \lesssim 1$. Then

\begin{align}
X_{nl} \lesssim \alpha N^{1-s} \cdot
\label{Eqn:Boundalphi}
\end{align}

\label{Prop:EstLinNonlin}
\end{prop}

\begin{rem}
It was proved in Proposition \ref{Prop:LocalBoundedness} that for $J$ satisfying
(\ref{Eqn:SizeJ}) we have

\begin{equation}
\begin{array}{l}
\max{ \left(
\begin{array}{l}
\| P_{\gtrsim N} \langle D \rangle^{1 - \frac{3}{8}} I u_{k}^{J} \|_{L_{t}^{8} L_{x}^{4} (J)},  \| P_{\gtrsim N} \partial_{t}
\langle D \rangle^{-\frac{3}{8}} I u_{k}^{J} \|_{L_{t}^{8} L_{x}^{4} (J)} \\
\| P_{\gtrsim N}  \langle D \rangle^{1 - \frac{1}{4}} I u_{k}^{J} \|_{L_{t}^{12} L_{x}^{3}(J)},
\| P_{\gtrsim N}  \langle D \rangle^{-\frac{1}{4}} \partial_{t}  I u_{k}^{J} \|_{L_{t}^{12} L_{x}^{3}(J)}
\end{array}
 \right) } \lesssim N^{1-s}.
\end{array}
\nonumber
\end{equation}
Hence, combining this estimate with (\ref{Eqn:Boundalphillin}), we see that
$X_{nl} \lesssim N^{1-s}$ for $J$ satisfying (\ref{Eqn:SizeJ}). So we see from
(\ref{Eqn:Boundalphi}) that the gain of regularity holds for all $J$ such that $\alpha \lesssim 1$.
\label{Rem:Gain}
\end{rem}

\begin{proof}

Let $m \in \left\{ \frac{1}{4}, \frac{3}{8} \right\}$. \\
Plugging $\langle D \rangle^{1-m} I$ into (\ref{Eqn:StrichWave}) (and using (\ref{Eqn:AprBdMolNrj})) we easily see that (\ref{Eqn:Boundalphillin}) holds. \\
Next we prove (\ref{Eqn:Boundalphi}). \\
Let $u_{nl}^{J,1}:= \Box^{-1} \left( P_{\gtrsim N} u (P_{ \lesssim N} u)^{2} \right)$,
$ u_{nl}^{J,2}:= \Box^{-1} \left( (P_{ \gtrsim N} u)^{2} P_{\lesssim N} u \right) $ and
$ u_{nl}^{J,3}:= \Box^{-1} \left( (P_{ \gtrsim N} u)^{3} \right) $
such that for $p \in \{1,2,3\}$, $ ( u_{nl}^{J,p}(a), \partial_{t} u_{nl}^{J,p}(a)
) := (0,0)$. Write $ P_{\gtrsim N} u_{nl}^{J} = \sum \limits_{p=1}^{3} u_{nl}^{J,p} $.  \\
By (\ref{Eqn:StrichWaveGain}) and (\ref{Eqn:KinetMassBd1}) we have

\begin{equation}
\begin{array}{l}
\max{ ( \| \partial_{t} \langle D \rangle^{1-m}  I u^{J,1}_{nl} \|_{L_{t}^{\frac{3}{m}} L_{x}^{\frac{6}{3-4m}} (J)}, \| \langle D \rangle^{2-m } I u^{J,1}_{nl} \|_{L_{t}^{\frac{3}{m}} L_{x}^{\frac{6}{3-4m}} (J)} )} \\
\lesssim \left\| \langle D \rangle I ( P_{\gtrsim N} u ( P_{\lesssim N } u)^{2} ) \right\|_{L_{t}^{1} L_{x}^{2} (J)} \\
\lesssim |J| \left(
\begin{array}{l}
\| \langle D \rangle I P_{\gtrsim N} u \|_{L_{t}^{\infty} L_{x}^{2} (J)} \| P_{\lesssim N} u \|^{2}_{L_{t}^{\infty} L_{x}^{\infty} (J)} \\
+ \| \langle D \rangle I P_{\lesssim N} u \|_{L_{t}^{\infty} L_{x}^{\infty}(J)} \| P_{\lesssim N} u \|_{L_{t}^{\infty} L_{x}^{\infty}(J)}
\| P_{\gtrsim N} u \|_{L_{t}^{\infty} L_{x}^{2}(J)}
\end{array}
\right) \\
\lesssim |J| N^{3(1-s)+},
\end{array}
\label{Eqn:HFreq1LFreq2}
\end{equation}
where we use

\begin{equation}
\begin{array}{l}
 \| P_{\gtrsim N} u \|_{L_{t}^{\infty} L_{x}^{2}(J)} \; \lesssim \sum_{M \gtrsim N} \frac{\| \langle D \rangle I u \|_{L_{t}^{\infty} L_{x}^{2}(J)}}{M} \; \lesssim
\frac{N^{1-s}}{N}, \\
\| \langle D \rangle I P_{\lesssim N } u \|_{L_{t}^{\infty} L_{x}^{\infty}(J)} \; \lesssim
\sum_{0 \leq M \lesssim N} \| \langle D \rangle P_{M} u \|_{L_{t}^{\infty} L_{x}^{\infty}(J)}
\; \lesssim N N^{1-s},
\end{array}
\nonumber
\end{equation}
and

\begin{equation}
\begin{array}{l}
\| P_{\lesssim N} u \|_{L_{t}^{\infty} L_{x}^{\infty} (J)} \;  \lesssim
\sum_{ 0 \leq M \lesssim N} \| P_{M} u \|_{L_{t}^{\infty} L_{x}^{\infty} (J)}
\; \lesssim  N^{+} \| \langle D \rangle I u \|_{L_{t}^{\infty} L_{x}^{2} (J)}.
\end{array}
\nonumber
\end{equation}
By (\ref{Eqn:StrichWave}), Proposition \ref{Prop:LocalBoundedness}, and
(\ref{Eqn:BoundPot}), we have

\begin{equation}
\begin{array}{l}
\max{ ( \| \partial_{t}  \langle D \rangle^{-m} I u_{nl}^{J,3} \|_{L_{t}^{\frac{3}{m}} L_{x}^{\frac{6}{3-4m}}(J)},
\|  \langle D \rangle^{1-m} I u_{nl}^{J,3} \|_{L_{t}^{\frac{3}{m}} L_{x}^{\frac{6}{3-4m}}(J)} )}  \\
\lesssim   \| \langle D \rangle^{1-m} I ( (P_{ \gtrsim N} u)^{3}) \|_{ L_{t}^{\frac{3}{2+m}} L_{x}^{\frac{6}{7-4m}}(J)}  \\
\lesssim  \| \langle D \rangle^{1-m} I P_{ \gtrsim N} u \|_{L_{t}^{\frac{3}{m}} L_{x}^{\frac{6}{3-4m}} (J) }
 \| P_{ \gtrsim N} u \|^{2}_{L_{t}^{3} L_{x}^{3}(J)} \\
\lesssim |J|^{\frac{1}{2}} (X_{l} + X_{nl}) \frac{  \| \langle D \rangle^{ 1 -\frac{1}{4}} I P_{ \gtrsim N} u \|^{2}_{L_{t}^{12} L_{x}^{3}(J)} }{N^{\frac{3}{2}}} \\
\lesssim |J|^{\frac{1}{2}} \frac{ (X_{l} + X_{nl})^{3}}{N^{\frac{3}{2}}} \\\
\lesssim |J|^{\frac{1}{2}} \frac{N^{3(1-s)}}{N^{\frac{3}{2}}} + |J|^{\frac{1}{2}} \frac{X_{nl}^{3}}{N^{\frac{3}{2}}} \cdot
\end{array}
\nonumber
\end{equation}

\begin{rem}

\begin{itemize}

\item Notice that we use dispersive estimates since we work with the high frequency part. This is consistent with
the second point of Remark \ref{Rem:LowNNHighN}.

\item Notice also that when we estimate the nonlinearity we decompose $u$ into its linear part and its nonlinear part: this allows
to use (\ref{Eqn:Boundalphillin}).

\end{itemize}

\end{rem}


\begin{equation}
\begin{array}{l}
\max{ \left( \| \partial_{t} \langle D \rangle^{-m} I u_{nl}^{J,2} \|_{L_{t}^{\frac{3}{m}} L_{x}^{\frac{6}{3-4m}}(J)},
\| \langle D \rangle^{1-m} I u_{nl}^{J,2} \|_{L_{t}^{\frac{3}{m}} L_{x}^{\frac{6}{3-4m}}(J)} \right)}  \\
\lesssim \| \langle D \rangle^{1-m} I  (P_{\lesssim N} u  (P_{\gtrsim N} u)^{2}  \|_{L_{t}^{\frac{3}{2+m}} L_{x}^{\frac{6}{7-4m}}(J)} \\
\lesssim
|J|^{\frac{7}{12}} \| \langle D \rangle^{1-m} I P_{\gtrsim N} u \|_{L_{t}^{\frac{3}{m}} L_{x}^{\frac{6}{3-4m}} (J)}
\| P _{\gtrsim N} u \|_{ L_{t}^{12} L_{x}^{3} (J)} \| P_{\lesssim N} u  \|_{L_{t}^{\infty} L_{x}^{3} (J)} \\
+ |J|^{\frac{1}{2}} \| \langle D \rangle^{1-m} I P_{\lesssim N} u \|_{L_{t}^{\frac{3}{m}} L_{x}^{\frac{6}{3-4m}} (J)}
\| P _{\gtrsim N} u \|^{2}_{ L_{t}^{12} L_{x}^{3} (J)}
 \\
\lesssim
|J|^{\frac{7}{12}}  \| \langle D \rangle^{1-m} I P_{\gtrsim N} u \|_{L_{t}^{\frac{3}{m}} L_{x}^{\frac{6}{3-4m}} (J)}
\frac{\| \langle D \rangle^{1- \frac{1}{4}} I P_{\gtrsim N} u \|_{L_{t}^{12} L_{x}^{3} (J)}}{N^{\frac{3}{4}}} \| I u \|_{L_{t}^{\infty} L_{x}^{4} (J)} \\
+ |J|^{\frac{1}{2}} \|  \langle D \rangle^{1-m} I P_{\lesssim N} u \|_{L_{t}^{\frac{3}{m}} L_{x}^{\frac{6}{3-4m}} (J)}
\frac{ \| \langle D \rangle^{1- \frac{1}{4}} I P_{\gtrsim N} u \|^{2}_{L_{t}^{12} L_{x}^{3} (J)} }{N^{\frac{3}{2}}} \\
\lesssim
|J|^{\frac{7}{12}} \frac{(X_{l} + X_{nl})^{2}}{N^{\frac{3}{4}}} N^{\frac{1-s}{2}}
+ |J|^{\frac{1}{2}} N^{1-s} \langle |J| N^{1-s} \rangle^{\frac{1}{8}}   \frac{(X_{l} + X_{nl})^{2}}{N^{\frac{3}{2}}}
 \\
\lesssim  |J|^{\frac{7}{12}} \frac{N^{\frac{5(1-s)}{2}}}{N^{\frac{3}{4}}} +
|J|^{\frac{7}{12}} \frac{N^{\frac{1-s}{2}} }{N^{\frac{3}{4}}} X^{2}_{nl}
+ |J|^{\frac{1}{2}} \frac{N^{3(1-s)} \langle |J| N^{1-s} \rangle^{\frac{1}{8}}}{N^{\frac{3}{2}}}
+ |J|^{\frac{1}{2}} \frac{ N^{1-s} \langle |J| N^{1-s} \rangle^{\frac{1}{8}}}{N^{\frac{3}{2}}} X_{nl}^{2}.
\end{array}
\label{Eqn:LowNHighNN}
\end{equation}
Therefore (\ref{Eqn:Boundalphi}) holds.

\end{proof}
Since we are interested in using this gain of regularity on an interval $J$ as large as possible, the maximal size of $J$
for which this gain holds is $ |\bar{J}| := \max_{|J| < \epsilon_{0}: \, \alpha \lesssim 1} |J|$. One has for $N \gg 1$

\begin{equation}
\begin{array}{l}
|\bar{J}|  \sim \alpha^{-1}(1)  \sim  min_{1 \leq k \leq 4} \, \alpha_{k} ^{-1}(1)
\sim \alpha_{1}^{-1}(1) \sim N^{(2s-1)-}:
\end{array}
\label{Eqn:ProcJtde}
\end{equation}
in particular, the gain of regularity holds for intervals $J$ with size larger
than (\ref{Eqn:SizeJ}).


\begin{rem}
Notice that we did not interpolate with the potential bound in (\ref{Eqn:HFreq1LFreq2}). If we had used the potential bound
for (\ref{Eqn:HFreq1LFreq2}) we would have found (following the same steps as (\ref{Eqn:EstLowNNHighN}))

\begin{equation}
\begin{array}{l}
\max { \left( \| \partial_{t} \langle D \rangle^{-m} I u_{nl}^{J,1} \|_{L_{t}^{\frac{3}{m}} L_{x}^{\frac{6}{3-4m}}(J)},
\| \langle D \rangle^{1-m} I u_{nl}^{J,1} \|_{L_{t}^{\frac{3}{m}} L_{x}^{\frac{6}{3-4m}}(J)} \right)  } \\
\lesssim \| \langle D \rangle^{1-m} I ( (P_{ \lesssim N} u)^{2} P_{ \gtrsim N} u \|_{L_{t}^{1} L_{x}^{\frac{2}{2-m}}(J)} \\
\lesssim |J|N^{2(1-s)} + |J|^{\frac{11}{12}} \frac{N^{\frac{8(1-s)}{3}}}{N^{\frac{3}{4}}} \cdot
\end{array}
\nonumber
\end{equation}
Hence letting

\begin{equation}
\begin{array}{ll}
\bar{\alpha} & := \max \left( |J| N^{1-s}, |J|^{\frac{11}{12}} \frac{N^{\frac{5(1-s)}{3}}}{N^{\frac{3}{4}}}, |J|^{\frac{1}{2}} \frac{N^{2(1-s)}}{N^{\frac{3}{2}}}, |J|^{\frac{7}{12}} \frac{N^{\frac{3(1-s)}{2}}}{N^{\frac{3}{4}}}, |J|^{\frac{1}{2}} \frac{N^{2(1-s)} \langle |J| N^{1-s} \rangle^{\frac{1}{8}}}{N^{\frac{3}{2}}}
\right) \\
& := \max ( \bar{\alpha}_{1}, \bar{{\bar{\alpha}}}_{1}, \alpha_{2}, \alpha_{3}, \alpha_{4} )
\end{array}
\nonumber
\end{equation}
and assuming that $\bar{\alpha} \lesssim 1$, we would have found $X_{nl} \lesssim \bar{\alpha} N^{1-s}$. By following the same procedure
as (\ref{Eqn:ProcJtde}), we would have found

\begin{equation}
\begin{array}{l}
\max_{|J| < \epsilon_{0}: \alpha \lesssim 1} |J| \sim \bar{\alpha}_{1}^{-1}(1) \sim N^{s-1}:
\end{array}
\nonumber
\end{equation}
this size is smaller than $N^{(2s-1)-}$. It occurs that for (\ref{Eqn:HFreq1LFreq2}) it is better to use the Strichartz estimates with gain of derivative and the kinetic bound than the potential bound.

\end{rem}

\begin{rem}
However, the use of the potential bound (\ref{Eqn:BoundPot}) in (\ref{Eqn:LowNHighNN}) is crucial. Indeed recall that the maximal size for which the gain of regularity holds is $ \max_{|J| < \epsilon_{0}: \alpha \lesssim 1} |J| $. If we had only used
(\ref{Eqn:BoundPotNrj}) we would have found

\begin{equation}
\begin{array}{l}
\max{ \left( \| \partial_{t} \langle D \rangle^{-m} I u_{nl}^{J,2} \|_{L_{t}^{\frac{3}{m}} L_{x}^{\frac{6}{3-4m}}(J)},
\|  \langle D \rangle^{1-m} I u_{nl}^{J,2} \|_{L_{t}^{\frac{3}{m}} L_{x}^{\frac{6}{3-4m}}(J)} \right)}  \\
\lesssim   |J|^{\frac{7}{12}} \frac{N^{3(1-s)}}{N^{\frac{3}{4}-}} +
|J|^{\frac{7}{12}} \frac{N^{1-s} }{N^{\frac{3}{4}}} X^{2}_{nl}
+ |J|^{\frac{1}{2}} \frac{N^{3(1-s)} \langle |J| N^{s-1} \rangle^{\frac{1}{8}}}{N^{\frac{3}{2}}}
+ |J|^{\frac{1}{2}} \frac{\langle |J| N^{s-1} \rangle^{\frac{1}{8}}}{N^{\frac{3}{2}}} X_{nl}^{2}:
\end{array}
\nonumber
\end{equation}
Hence letting

\begin{equation}
\begin{array}{ll}
\bar{\alpha} & := \max \left( \frac{|J| N^{2(1-s)+}}{N}, |J|^{\frac{1}{2}} \frac{N^{2(1-s)}}{N^{\frac{3}{2}}}, |J|^{\frac{7}{12}}
\frac{N^{2(1-s)}}{N^{\frac{3}{4}-}}, |J|^{\frac{1}{2}} \frac{N^{2(1-s)} \langle |J| N^{1-s} \rangle^{\frac{1}{8}}}{N^{\frac{3}{2}}}
\right). \\
& := \max( \alpha_{1},\alpha_{2}, \bar{\alpha}_{3},\alpha_{4} ),
\end{array}
\nonumber
\end{equation}
we would have found $X_{nl} \lesssim \bar{\alpha} N^{1-s}$. By following the same procedure as (\ref{Eqn:ProcJtde}) we would have found

\begin{equation}
\begin{array}{l}
\max_{|J| < \epsilon_{0}: \alpha \lesssim 1} |J| \sim \bar{\alpha}_{3}^{-1}(1) \sim N^{\frac{3}{7}(8s-5)}:
\end{array}
\nonumber
\end{equation}
this size is smaller than $N^{(2s-1)-}$.
\label{Rem:InterpAgain}
\end{rem}

\subsection{Variation of $E(Iu)$ on $J \subset [0,\epsilon_{0}]$}
\label{Subsec:VarNrjMolAdapt}

In this subsection we estimate the variation of the mollified energy on $J$ such that $|J| \lesssim N^{(2s-1)-}$.

\begin{prop}
Let $J \subset [0,\epsilon_{0}]$ with size $ \lesssim N^{(2s-1)-} $. Let $s > \frac{3}{8}$. The variation of $E(Iu)$ on $J$ is given by

\begin{equation}
Var(E(Iu),J) \lesssim \frac{|J|^{\frac{3}{4}} N^{4(1-s)}}{N^{\frac{5}{4}-}}
+ \frac{ |J|^{\frac{1}{2}} N^{4(1-s)}}{N^{\frac{3}{2}-}} \cdot
\label{Eqn:EstMolNrjAdapt}
\end{equation}

\end{prop}

\begin{proof}

From (\ref{Eqn:Boundalphillin}), (\ref{Eqn:Boundalphi}), and the decomposition $u=u_{l}^{J} + u_{nl}^{J}$, we see that

\begin{equation}
\max{ \left(
\begin{array}{l}
\| P_{\gtrsim N} \langle D \rangle^{1 - \frac{3}{8}} I u \|_{L_{t}^{8} L_{x}^{4} (J)},  \| P_{\gtrsim N} \partial_{t}
\langle D \rangle^{-\frac{3}{8}} I u \|_{L_{t}^{8} L_{x}^{4} (J)} \\
\| P_{\gtrsim N}  \langle D \rangle^{1 - \frac{1}{4}} I u \|_{L_{t}^{12} L_{x}^{3}(J)},
\| P_{\gtrsim N}  \langle D \rangle^{-\frac{1}{4}} \partial_{t}  I u \|_{L_{t}^{12} L_{x}^{3}(J)}
\end{array}
 \right) }
\lesssim N^{1-s} \cdot
\nonumber
\end{equation}
Hence plugging this bound into the proof of Proposition \ref{Prop:EstVarNrj}, we get (\ref{Eqn:EstMolNrjAdapt}).

\end{proof}

\subsection{Total variation of $E(Iu)$ and global existence for $s > \frac{2}{5}$ }

Partitioning $[0,\epsilon_{0}]$ into subintervals $J$ of size $ \sim N^{(2s-1)-}$  (except maybe the last one), we get by iteration

\begin{equation}
\begin{array}{ll}
\sup_{t \in [0,\epsilon_{0}]} E(Iu(t))  & \lesssim  \frac{1}{|J|}
\left(
\begin{array}{l}
\frac{|J|^{\frac{3}{4}} N^{4(1-s)}}{N^{\frac{5}{4}-}} + \frac{ |J|^{\frac{1}{2}} N^{4(1-s)}}{N^{\frac{3}{2}-}}
\end{array}
\right). \\
\end{array}
\label{Eqn:SupNrjMol25}
\end{equation}
Hence we see that we must assume $s > \frac{2}{5}$ in order to choose $N \gg 1$ so that (\ref{Eqn:EstNrjGoal}) holds. This proves
global existence for $s > \frac{2}{5}$.\\
\\






\textbf{Acknowledgments}\\
\\
The author would like to thank Nikolay Tzvetkov for suggesting him this problem and for valuable discussions related to this
work. The author was supported by the ERC Advanced Grant no. 291214, BLOWDISOL, while he worked on this problem at Universit\'e de Cergy-Pontoise.
This manuscript is available on http://arxiv.org/abs/1411.6141.

\appendix
\section{}


\subsection{Leibnitz rule}

We prove a fractional Leibnitz rule on the torus. The proof is essentially well-known in the
literature in the euclidean space (see e.g \cite{christweins,KatoPonce,KenigPonceVega,taylor}). By using transference
multipliers results we can adapt it to the torus. We give a proof for the convenience of
the reader.

\begin{prop}
Let $s<1$, $\alpha \geq 1-s$, and $u_{1}, u_{2}, u_{3}$ be three smooth functions on $\mathbb{T}^{2}$. Then

\begin{equation}
\begin{array}{ll}
\| \langle D \rangle^{\alpha} I (u_{1}  u_{2} u_{3}) \|_{L^{r} (\mathbb{T}^{2})} & \lesssim
\| \langle D \rangle^{\alpha} I u_{1} \|_{L^{r_{1}}(\mathbb{T}^{2})} \| u_{2} \|_{L^{r_{2}} (\mathbb{T}^{2})} \| u_{3} \|_{L^{r_{3}} (\mathbb{T}^{2})} \\
& + \| u_{1} \|_{L^{\bar{r}_{1}} (\mathbb{T}^{2})} \| \langle D \rangle^{\alpha} I u_{2} \|_{L^{\bar{r}_{2}}(\mathbb{T}^{2})} \| u_{3} \|_{L^{\bar{r}_{3}}(\mathbb{T}^{2})} \\
& + \| u_{1} \|_{L^{\bar{\bar{r}}_{1}} (\mathbb{T}^{2})} \| u_{2} \|_{L^{\bar{\bar{r}}_{2}}(\mathbb{T}^{2})}
\| \langle  D \rangle^{\alpha} I u_{3} \|_{L^{\bar{\bar{r}}_{3}}(\mathbb{T}^{2})}
\end{array}
\nonumber
\end{equation}
under the conditions

\begin{equation}
\begin{array}{l}
 (r_{1},r_{2},r_{3}), \; (\bar{r}_2,\bar{r}_1,\bar{r}_3), \, \text{and} \; (\bar{\bar{r}}_3,\bar{\bar{r}}_1, \bar{\bar{r}}_2) \in (1,\infty) \times (1, \infty]^{2}, \; \text{and} \\
\sum \limits_{m=1}^{3} \frac{1}{r_{m}}  = \sum \limits_{m=1}^{3}
\frac{1}{\bar{r}_{m}} = \sum \limits_{m=1}^{3}
\frac{1}{_{\bar{\bar{r}}_{m}}} = \frac{1}{r} \cdot
\end{array}
\label{Eqn:CondFrac}
\end{equation}

\end{prop}

\begin{proof}

It is sufficient to prove
\begin{equation}
\| \langle D \rangle^{\alpha} I (u_{1} u_{2}) \|_{L^{r} (\mathbb{T}^{2})} \lesssim \| \langle D \rangle^{\alpha} I u_{1} \|_{L^{r_{1}} (\mathbb{T}^{2})}
\| u_{2} \|_{L^{r_{2}} (\mathbb{T}^{2})} + \| u_{1} \|_{L^{\bar{r}_{1}} (\mathbb{T}^{2})}
\| \langle D \rangle^{\alpha}  I u_{2} \|_{L^{\bar{r}_{2}} (\mathbb{T}^{2})}
\label{Eqn:FracLeibTwo}
\end{equation}
with $(r_{1},r_{2}, \infty)$ and $(\bar{r}_2, \bar{r}_1, \infty)$ satisfying (\ref{Eqn:CondFrac}).\\
\\
We first claim that  if $m \geq 0$ then

\begin{equation}
\| \langle D \rangle^{m} (u_1 u_2) \|_{L^{r}(\mathbb{T}^{2})} \lesssim \| \langle D \rangle^{m} u_1 \|_{L^{r_1}(\mathbb{T}^{2})}
\| u_2 \|_{L^{r_2}(\mathbb{T}^{2})} + \| u_1 \|_{L^{\bar{r}_1}(\mathbb{T}^{2})} \| \langle D \rangle^{m} u_2 \|_{L^{\bar{r}_2}(\mathbb{T}^{2})} \cdot
\label{Eqn:FracLeibnProd}
\end{equation}
Indeed if $m=0$ then it follows from H\"older inequality. So we may assume that $m >0$. In the sequel if $T$ is an operator then
$\widetilde{T}$ denotes an operator that behaves like $T$ in the frequency space and whose definition is allowed to change from one line line and even in the same line. The Bony decomposition yields

\begin{equation}
\begin{array}{ll}
u_1 u_2 & = \widetilde{P_0} u_1 \widetilde{P_0} u_2 + \widetilde{P_0} u_2 \sum \limits_{M > 16} P_{M} u_1
+ \widetilde{P_0} u_1  \sum \limits_{M > 16} P_{M} u_2
+ \sum \limits_{M > 16} P_M u_1 P_{ 16  < \cdot < \frac{M}{16}} u_2  \\
& + \sum \limits_{M > 16} P_M u_2 P_{ 16  < \cdot < \frac{M}{16}} u_1
+ \sum \limits_{\substack{M_1, M_2 > 16  \\ M_1 \sim M_2}} P_{M_1} u_1  P_{M_2} u_2 \\
& = W + X_1 + X_2 + Y_1 + Y_2 + Z
\end{array}
\nonumber
\end{equation}
with $W:= \widetilde{P_0} u_1 \widetilde{P_0} u_2 + \widetilde{P_1} u_1 \widetilde{P_0} u_2 + \widetilde{P_1} u_1 \widetilde{P_1} u_2$,
$ X_1 := \widetilde{P_0} u_2 \sum \limits_{M \gg 1} P_{M} u_1 $, \\
$ X_2 := \widetilde{P_0} u_1 \sum \limits_{M \gg 1} P_{M} u_2 $,
$Y_1 := \sum \limits_{M > 16} P_M u_1 P_{ 16  < \cdot < \frac{M}{16}} u_2  $, \\
$Y_2 := \sum \limits_{M > 16} P_M u_2 P_{ 16  < \cdot < \frac{M}{16}} u_1 $,
and $Z := \sum \limits_{\substack{M_1,M_2 > 16 \\ M_1 \sim M_2}} P_{M_1} u_1  P_{M_2} u_2 $.\\
\\
We only deal with the first term of $W$: the other terms are treated similarly. Writing
$ \widetilde{P_0} u_1  \widetilde{P_0} u_2 = \widetilde{P_0} (\widetilde{P_0} u_1 \widetilde{P_0} u_2) $, applying Bernstein-type
inequalities and H\"older inequality, we can bound it by the right-hand side of (\ref{Eqn:FracLeibnProd}). \\
We then consider the decomposition $X_1 +.... + Z$ in the euclidean space $\mathbb{R}^{2}$. We make the following notation:
the functions, operators, etc. are indexed by ``prime'' whenever we consider them as functions, operators, etc. in the
euclidean space. For example $\widehat{P_{0}^{'} f'}(\xi) = \phi(\xi) \widehat{f'}(\xi) $ where $\widehat{f'}(\xi)$
denotes the Fourier transform in $\mathbb{R}^{2}$ of a Schwartz function $f'$ on $\mathbb{R}^{2}$. Recall (see \cite{taylor}, p 105) that

\begin{equation}
\begin{array}{l}
\| \langle D \rangle^{m} X_1^{'} \|_{L^{r}(\mathbb{R}^{2})}, \, \| \langle D \rangle^{m} Y_1^{'} \|_{L^{r}(\mathbb{R}^{2})}
\lesssim  \| \langle D \rangle^{m} u_{1}^{'} \|_{L^{r_1}(\mathbb{R}^{2})}
\| u_{2}^{'} \|_{L^{r_2}(\mathbb{R}^{2})} \\
\| \langle D \rangle^{m} X_2^{'} \|_{L^{r}(\mathbb{R}^{2})}, \, \| \langle D \rangle^{m} Y_2^{'} \|_{L^{r}(\mathbb{R}^{2})}
\lesssim  \|  u_{1}^{'} \|_{L^{\bar{r}_1}(\mathbb{R}^{2})}
\| \langle D \rangle^{m} u_{2}^{'} \|_{L^{\bar{r}_2}(\mathbb{R}^{2})}\\
\| \langle D \rangle^{m} Z^{'} \|_{L^{r}(\mathbb{R}^{2})}  \lesssim \| \langle D \rangle^{m} u_{1}^{'} \|_{L^{r_1}(\mathbb{R}^{2})}
\| u_{2}^{'} \|_{L^{r_2}(\mathbb{R}^{2})}
\end{array}
\nonumber
\end{equation}
Hence, by Theorem 3, p 39 of \cite{fansato}, we see that all these estimates can be transferred to the torus. For example
$ \| \langle D \rangle^{m} X_1 \|_{L^{r}(\mathbb{T}^{2})} \lesssim  \| \langle D \rangle^{m} u_{1} \|_{L^{r_1}(\mathbb{T}^{2})}
\| u_{2} \|_{L^{r_2}(\mathbb{T}^{2})} $.\\
\\
We then write $ \langle D \rangle^{\alpha} I (u_{1} u_{2}) = \langle D \rangle^{\alpha} I ( X_1 + X_2 + X_3 + X_4) $ with

\begin{equation}
\begin{array}{l}
X_1:=  P_{\lesssim N} u_{1} P_{\lesssim N} u_{2}; \;
X_2:= P_{\gtrsim N} u_1 P_{\lesssim N} u_2; \\
X_3 :=  P_{\lesssim N} u_1 P_{\gtrsim N} u_2; \; \text{and} \;
X_4 :=  P_{\gtrsim N} u_1 P_{\gtrsim N} u_2 \cdot
\end{array}
\nonumber
\end{equation}
Write  $ X_1 = \widetilde{P_{\lesssim N}} (  P_{\lesssim N} u_{1} P_{\lesssim N} u_{2})$. We see from
Bernstein-type inequalities and (\ref{Eqn:FracLeibnProd}) that
$ \| \langle D \rangle^{\alpha} I  X_1 \|_{L^{r} (\mathbb{T}^{2})}$ is bounded by the right-hand side of
(\ref{Eqn:FracLeibTwo}).\\
We then turn to $X_{2}$ and $X_{3}$. By symmetry it is sufficient to estimate
$ \| \langle D \rangle^{\alpha} I  X_{2}\|_{L^{r} (\mathbb{T}^{2})}$. Write
$ X_2 = X_{2,1} + X_{2,2}$ with $  X_{2,1} := \widetilde{P_{N}} u_{1} P_{\lesssim N} u_{2} $
and $X_{2,2} := \widetilde{P_{\gg N}} u_{1} P_{\lesssim N} u_{2} $. Again from Bernstein-type inequalities and
(\ref{Eqn:FracLeibnProd}) we see that
$ \| \langle D \rangle^{\alpha} I  X_{2,1}\|_{L^{r} (\mathbb{T}^{2})}$ is bounded by the right-hand side of
(\ref{Eqn:FracLeibTwo}). We then turn to estimating $X_{2,2}$. We have

\begin{equation}
\begin{array}{ll}
\| \langle D \rangle^{\alpha} I  X_{2,2} \|_{L^{r} (\mathbb{T}^{2})} & \lesssim
\left\| \left( \sum \limits_{M \gg N} | N^{1-s} M^{\alpha -(1-s)} P_{M} u_1  P_{\lesssim N} u_2|^{2} \right)^{\frac{1}{2}} \right\|_{L^{r}(\mathbb{T}^{2})} \\
& \lesssim \| \langle D \rangle^{\alpha} I u_1 \|_{L^{r_1} (\mathbb{T}^{2})} \| u_2 \|_{L^{r_2}(\mathbb{T}^{2})},
\end{array}
\nonumber
\end{equation}
We then turn to $X_{4}$. Write $X_{4}= X_{4,1} + X_{4,2}$ with
$X_{4,1} := P_{\lesssim N} \left( P_{\gtrsim N} u_1 P_{\gtrsim N} u_2 \right) $
and $X_{4,2} := P_{\gtrsim N} \left( P_{\gtrsim N} u_1 P_{\gtrsim N} u_2 \right)$.
Again, by Bernstein-type inequalities and (\ref{Eqn:FracLeibnProd}) we see that
$ \| \langle D \rangle^{\alpha} I  X_{4,2}\|_{L^{r} (\mathbb{T}^{2})}$ is bounded by the right-hand side of
(\ref{Eqn:FracLeibTwo}). Applying the Bony decomposition to $P_{\gtrsim N} u_1 P_{\gtrsim N} u_2$ we get

\begin{equation}
\begin{array}{ll}
P_{\lesssim N} \left( P_{\gtrsim N} u_1 P_{\gtrsim N} u_2 \right) & =  \sum \limits_{M \sim N} P_{\lesssim N}
(  P_M u_1 P_{N \lesssim \cdot < \frac{M}{16}} u_2 )
+ \sum \limits_{M \sim N} P_{\lesssim N} ( P_{ N  \lesssim \cdot < \frac{M}{16}} u_1 P_{M} u_2 ) \\
& + \sum \limits_{M_1 \sim M_2 \gtrsim N} P_{\lesssim N} ( P_{M_1} u_1  P_{M_2} u_2 )  \\
& = \bar{Y}_{1} + \bar{Y}_{2} + \bar{Z}
\end{array}
\nonumber
\end{equation}
Bernstein-type inequalities and a straightforward modification of the arguments used in \cite{taylor}, p 105 show that

\begin{equation}
\begin{array}{l}
\| \langle D \rangle^{\alpha} I \bar{Y}_1^{'} \|_{L^{r}(\mathbb{R}^{2})}
\lesssim  \| \langle D \rangle^{\alpha} I u_{1}^{'} \|_{L^{r_1}(\mathbb{R}^{2})}
\| u_{2}^{'} \|_{L^{r_2}(\mathbb{R}^{2})} \\
| \langle D \rangle^{\alpha} I \bar{Y}_2^{'} \|_{L^{r}(\mathbb{R}^{2})}
\lesssim  \| u_{1}^{'} \|_{L^{\bar{r}_1}(\mathbb{R}^{2})}
\| \langle D \rangle^{\alpha} I  u_{2}^{'} \|_{L^{\bar{r}_2}(\mathbb{R}^{2})}\\
\| \langle D \rangle^{\alpha} I \bar{Z}^{'} \|_{L^{r}(\mathbb{R}^{2})}  \lesssim \| \langle D \rangle^{\alpha} I u_{1}^{'} \|_{L^{r_1}(\mathbb{R}^{2})}
\| u_{2}^{'} \|_{L^{r_2}(\mathbb{R}^{2})}
\end{array}
\nonumber
\end{equation}
Hence, by Theorem 3, p 39 of \cite{fansato}, we see that all these estimates can be transferred to the torus.

\end{proof}

\subsection{Procedure to estimate variation of mollified energy}
\label{Subsec:Procedure}

Recall the Coifman-Meyer theorem

\begin{thm}[\cite{coifmeyer}, p 178]
Let $k \in \mathbb{N}$ and $\sigma: \mathbb{R}^{2k} \rightarrow \mathbb{C}$ be an infinitely differentiable symbol such that for
all $\xi:= (\xi_{1},...\xi_{k}) \in \mathbb{R}^{2k}$ and for all $\alpha \in \mathbb{N}^{2k}$
\begin{equation}
|\partial_{\xi}^{\alpha} \sigma (\xi)| \lesssim_{\alpha} \frac{1}{(1 + |\xi|)^{|\alpha|}}.
\label{Eqn:Smoothmult}
\end{equation}
Then, defining
\begin{equation}
\Lambda(f_{1},...,f_{k})(x) := \int_{\mathbb{R}^{2k}}  \sigma(\xi_{1},...,\xi_{k}) \hat{f}_{1}(\xi_{1}).... \hat{f}_{k}(\xi_{k})
e^{2i \pi (\xi_{1} + ... + \xi_{k}) \cdot x}  \, d \xi_{1}....
d \xi_{k},
\nonumber
\end{equation}
we have

\begin{equation}
\| \Lambda(f_{1},...,f_{k}) \|_{L^{p}(\mathbb{R}^{2})} \lesssim \| f_{1} \|_{L^{p_{1}}(\mathbb{R}^{2})}... \| f_{k} \|_{L^{p_{k}}(\mathbb{R}^{2})},
\nonumber
\end{equation}
with $(p_{1},...,p_{k}) \in (1, \infty)^{k}$ such that $\frac{1}{p}= \frac{1}{p_{1}} +....+ \frac{1}{p_{k}} \leq 1$.
\end{thm}

\begin{rem}
The theorem also holds if (\ref{Eqn:Smoothmult}) is substituted for

\begin{equation}
|\partial_{\xi}^{\alpha} \sigma (\xi)| \lesssim_{\alpha} \prod \limits_{m=1}^{k} |\xi_{m}|^{- \alpha_{m}}.
\label{Eqn:Smoothmul2}
\end{equation}
See e.g \cite{muscalu}.

\end{rem}




Hence, from Theorem 3, p 39 of \cite{fansato} we see that we also have

\begin{equation}
\| \tilde{\Lambda}(f_{1},...,f_{k}) \|_{L^{p}(\mathbb{T}^{2})} \lesssim \| f_{1} \|_{L^{p_{1}}(\mathbb{T}^{2})} ... \| f_{k} \|_{L^{p_{k}}(\mathbb{T}^{2})},
\label{Eqn:CoifMeyTorus}
\end{equation}
with

\begin{equation}
\tilde{\Lambda}(f_{1},...,f_{k})(x) := \sum_{(n_{1},...,n_{k}) \in \mathbb{Z}^{2k}} \sigma(n_{1},..,n_{k})  \hat{f}_{1}(n_{1})...
\hat{f}_{k}(n_{k}) e^{ 2 i \pi (n_{1} + ... + n_{k}).x} \cdot
\nonumber
\end{equation}
Next we use an argument in \cite{scatttao}. More precisely define for $ k \in \{ \alpha,\beta \} $

\begin{equation}
\begin{array}{l}
\tilde{\Lambda}_{0,k}(f,g)(x) := \sum \limits_{(n_{1},n_{2}) \in (\mathbb{Z}^{2})^{2}}  \sigma_{0,k}(n_{1},n_{2})
\widehat{f}(n_{1})  \widehat{g}(n_{2})  e^{ 2 i \pi (n_{1}+n_{2}) \cdot x }, \\
j \in \{ 1,..,4\}: \; \tilde{\Lambda}_{j,k}(f)(x) := \sum \limits_{n_{j} \in \mathbb{Z}^{2}}  \sigma_{j,k}(n_{j})
\widehat{f}(n_{j}) e^{ 2 i \pi n_{j} \cdot x } \cdot
\end{array}
\nonumber
\end{equation}

\begin{itemize}

\item For Case $1.a$ we have $X_{\vec{M}} = X_{\alpha,\vec{M}} + X_{\beta,\vec{M}}$ with

\begin{equation}
k \in \{\alpha,\beta \}: \; X_{k,\vec{M}} = B \sum \limits_{n_{0} \in \mathbb{Z}^{2}} \int_{J}
\widehat{\tilde{\Lambda}}_{0,k} (\partial_{t} I u_1(t), I u_2(t))(n_{0})
\left( \overline{ \widehat{\tilde{\Lambda}}_{3,k}(I u_3(t)) \ast \widehat{\tilde{\Lambda}}_{4,k}(I u_4(t))} \right)  (n_0)
 \, dt,
\nonumber
\end{equation}
and $B$ being the right-hand side of (\ref{Eqn:MultCase1a}). Here

\begin{equation}
\begin{array}{l}
\sigma_{0,\alpha} (\xi_1,\xi_2) := \tilde{\phi}_{M_{1}}(\xi_{1})
\tilde{\phi}_{M_{2}}(\xi_{2}), \; \text{and} \\
j \in \{ 3,4 \}: \; \sigma_{j,\alpha} (\xi_j) := m(M_j) \tilde{\phi}_{M_{j}}(\xi_{j}), \\
\\
\sigma_{0,\beta} (\xi_1,\xi_2) := - \frac{m(\xi_1)}{m(\xi_2)} \tilde{\phi}_{M_{1}}(\xi_{1})  \tilde{\phi}_{M_{2}}(\xi_{2}), \; \text{and} \\
j \in \{ 3,4 \}: \; \sigma_{j,\beta} (\xi_j) :=  \frac{m(M_j)}{m(\xi_j)} \tilde{\phi}_{M_{j}}(\xi_{j}) \cdot
\end{array}
\nonumber
\end{equation}
 Here  $\tilde{\phi}_{M} (\xi) := \tilde{\phi} \left( \frac{\xi}{M} \right)$ if $M \in 2^{\mathbb{N}}$ and
$ \tilde{\phi}_{M} (\xi) := \tilde{\psi} (\xi)$ if $M =0$. \\
One can check that the symbols  satisfy (\ref{Eqn:Smoothmult}). By using Plancherel, H\"older, Young, and (\ref{Eqn:CoifMeyTorus}), we get (\ref{Eqn:ConclStrategy}).\\

\item For Case $2$ we write $ X_{\vec{M}} = X_{\alpha,\vec{M}} + X_{\beta,\vec{M}}$ with

\begin{equation}
\begin{array}{ll}
k \in \{ \alpha,\beta \}: \; X_{k,\vec{M}} & = B \sum \limits_{n_{0} \in \mathbb{Z}^{2}} \int_{J}
\widehat{\tilde{\Lambda}}_{0,k}(I u_2(t),I u_3 (t)) (n_0)
\left( \overline{ \widehat{\tilde{\Lambda}}_{1,k} (\partial_{t} I u_1(t)) \ast
\widehat{\tilde{\Lambda}}_{4,k}(I u_4(t))} \right)(n_0)
\, dt,
\end{array}
\nonumber
\end{equation}
with $B$ being the right-hand side of (\ref{Eqn:EstMultCase2}),

\begin{equation}
\begin{array}{l}
\sigma_{1,\alpha} (\xi_1) := \frac{m(M_2)}{m(M_1)} \tilde{\phi}_{M_{1}}(\xi_{1}), \; \\
\sigma_{0,\alpha} (\xi_2,\xi_3) := m(M_2) \tilde{\phi}_{M_{2}}(\xi_{2}) \tilde{\phi}_{M_{3}}(\xi_{3}), \; \text{and} \\
\sigma_{4,\alpha} (\xi_4) := m(M_4) \tilde{\phi}_{M_{4}}(\xi_{4})  \\
\\
\sigma_{1,\beta} (\xi_1) := - \frac{m(\xi_1)}{m(M_1)} \tilde{\phi}_{M_{1}}(\xi_{1})   \; \text{and} \\
\sigma_{0,\beta} (\xi_2,\xi_3) :=  \frac{m^{2}(M_2)}{m(\xi_2) m(\xi_3)} \tilde{\phi}_{M_{2}}(\xi_{2}) \tilde{\phi}_{M_{3}}(\xi_{3}), \; \text{and}  \\
\sigma_{4,\beta} (\xi_4) := \frac{m(M_4)}{m(\xi_4)} \tilde{\phi}_{M_{4}}(\xi_{4}) \cdot
\end{array}
\nonumber
\end{equation}
One can check that the symbols  satisfy (\ref{Eqn:Smoothmult}). Again by using Plancherel, H\"older, Young, and (\ref{Eqn:CoifMeyTorus}),
we get (\ref{Eqn:ConclStrategy}).\\
\\
\item For Case $1.b$ we define

\begin{equation}
\begin{array}{l}
\tilde{\Lambda}(f,g,h)(x) :=  \sum \limits_{(n_{2},n_{3},n_4) \in (\mathbb{Z}^{2})^{2}} \sigma(n_2,n_3,n_4) \widehat{f}(n_{2})\widehat{g}(n_{3}) \widehat{h}(n_{4}) e^{ 2 i \pi (n_2 + n_3 + n_4) \cdot x },
\end{array}
\nonumber
\end{equation}
with

\begin{equation}
\sigma(\xi_2,\xi_3,\xi_4) := B^{-1} \left( 1 - \frac{m(\xi_2+\xi_3+\xi_4)}{m(\xi_2)} \right)
\prod \limits_{l=2}^{4} \tilde{\phi}_{M_l}(\xi_l),
\nonumber
\end{equation}
and $B$ being the right-hand side of (\ref{Eqn:MultCase1b}). We have

\begin{equation}
X_{\vec{M}} = B \sum_{n_{0} \in \mathbb{Z}^{2}} \int_{J}
\overline{\widehat{\partial_{t} I u_1(t)}}(n_{0})
 \widehat{\tilde{\Lambda}} (I u_2(t), I u_3(t),I u_4(t))    (n_0)
 \, dt \cdot
 \nonumber
\end{equation}
One can check that the symbol satisfies (\ref{Eqn:Smoothmul2}). Again by using
Plancherel, H\"older, and (\ref{Eqn:CoifMeyTorus}), we get (\ref{Eqn:ConclStrategy}).

\end{itemize}



\end{document}